\documentclass[bj]{imsart}

\RequirePackage[OT1]{fontenc}
\RequirePackage{amsthm,amsmath,natbib}
\RequirePackage[colorlinks,citecolor=blue,urlcolor=blue]{hyperref}

\usepackage[T1]{fontenc}												
\usepackage[utf8x]{inputenc}											
\usepackage[english]{babel}												
\usepackage{amsmath,amsthm,amstext,amssymb,amsfonts}													
\usepackage{dsfont} 													
\usepackage{graphicx}													
\usepackage[bf,hang,nooneline,justification=centering]{caption}			

\usepackage{hyperref}													
\usepackage[titletoc,title]{appendix}									
\usepackage{lmodern} 													
\usepackage{xcolor}														
\usepackage{autonum}													
\usepackage[medium]{titlesec}											
\allowdisplaybreaks														

\usepackage[xcolor,pdftex,leftbars]{changebar}							
\cbcolor{black}

\arxiv{arXiv:2003.01016}

\startlocaldefs
\numberwithin{equation}{section}
\theoremstyle{plain}

\theoremstyle{remark}

\endlocaldefs

\newtheorem{theorem}{Theorem}[section]

\newtheorem{lemma}[theorem]{Lemma}

\newtheorem{proposition}[theorem]{Proposition}
\theoremstyle{definition} 								

\numberwithin{equation}{section}   						

\newcommand{\ind}[1]{\mathds{1}_{#1}}
\newcommand{\ceil}[1]{\lceil #1 \rceil}
\newcommand{\floor}[1]{\lfloor #1 \rfloor}

\def\N{{\mathbb{N}}}
\def\R{{\mathbb{R}}}
\def\Z{{\mathbb{Z}}}

\def\GG{{\mathcal{G}}}

\def\BB{{\mathcal{B}}}

\def\ord{{\hbox{o}}}
\def\Ord{{\hbox{O}}}

\begin{document}

\begin{frontmatter}
\title{Asymptotics for sliding blocks estimators of rare events}
\runtitle{Sliding blocks estimators}
\begin{aug}
  \author[A]{\fnms{Holger} \snm{Drees}\ead[label=e1]{drees@math.uni-hamburg.de}},
  \author[A]{\fnms{Sebastian} \snm{Neblung}\ead[label=e2,mark]{sebastian.neblung@uni-hamburg.de}}
  \address[A]{University of Hamburg, Department of Mathematics,
 SPST, Bundesstr.\ 55, 20146 Hamburg, Germany, \printead{e1,e2}}
\end{aug}

\begin{abstract}
  { } \cite{drees2010} have established limit theorems for a general class of empirical processes of statistics that are useful for the extreme value analysis of time series, but do not apply to statistics of sliding blocks, including so-called runs estimators. We generalize these results to empirical processes which cover both the class considered by \cite{drees2010} and processes of sliding blocks statistics. Using this approach, one can analyze different types of statistics in a unified framework. We show that statistics based on sliding blocks are asymptotically normal with an asymptotic variance which,  under rather mild conditions, is smaller than or equal to the asymptotic variance of the corresponding estimator based on disjoint blocks. Finally, the general theory is applied to three well-known estimators of the extremal index. It turns out that they all have the same limit distribution, a fact which has so far been overlooked in the literature.
\end{abstract}
\begin{keyword}[class=AMS]
\kwd[Primary ]{62G32}  \kwd[; secondary ]{62M10, 62G05, 60F17}
\end{keyword}

\begin{keyword}
\kwd{asymptotic efficiency, empirical processes, extremal index, extreme value analysis, sliding vs disjoint blocks,  time series, uniform central limit theorems}
\end{keyword}

\end{frontmatter}												


\section{Introduction}

The analysis of the serial dependence between large observations is crucial for a thorough understanding of the extreme value behavior of stationary time series. In the peaks over threshold (POT) approach, estimators of the dependence structure can usually be defined blockwise. To be more specific, assume that, starting from a stationary $\R^d$-valued time series $(X_t)_{1\le t\le n}$, \label{page.n} random variables (rv's) $X_{n,i}$ are defined, that in some sense capture its extreme value behavior. The most common example is $X_{n,i}:=(X_i/u_n)\ind{(u_n,\infty)}(\|X_i\|)$ for some threshold $u_n$ and some norm $\|\cdot\|$ on $\R^d$, but for certain applications $X_{n,i}$ may also depend on observations in the neighborhood of extreme observations. We consider statistics $g(Y_{n,j})$ of blocks
\begin{equation} \label{eq:Ydef}
  Y_{n,j} := (X_{n,j},\ldots,X_{n,j+s_n-1})
\end{equation}
of (possibly increasing) length $s_n$, starting with the $j$th rv. Estimators and test statistics of interest can then be defined in terms of averages of such blocks statistics.
For example, the well-known blocks estimator of the extremal index (roughly speaking, the reciprocal of the mean size of a cluster of extreme values) is of this type; see Section \ref{section:extremal.index} for details. Other examples are the empirical extremogram analyzed by \cite{davis2009}, forward and backward estimators of the distribution of the spectral tail process of a regularly varying time series examined by \cite{drees2015} and \cite{davis2018}, and the estimator of the cluster size distribution proposed by \cite{hsing1991}.

Here one may average either statistics $g(Y_{n,is_n+1})$, $0\le i\le \floor{n/s_n}-1$, of disjoint blocks or statistics $g(Y_{n,i})$, $1\le i\le n-s_n+1$, of overlapping sliding blocks. It has been suggested in the literature that the latter approach may often be more efficient; see, e.g., \cite{beirlant2004}, p.\ 390, for such a statement about blocks estimators of the extremal index. However, the asymptotic performance of both approaches has been compared only for a couple of estimators, while general results showing the superiority of the sliding blocks estimators are not yet known in the POT setting. \cite{robert2009} proved that for a different type of estimators of the extremal index the version using sliding blocks has a strictly smaller asymptotic variance than the one based on disjoint blocks, while the bias is asymptotically the same. In a block maxima setting, \cite{zou2019} proved that under quite general conditions an estimator of the extreme value copula of multivariate stationary time series is more efficient if it is based on sliding rather than disjoint blocks. The same observation has been made by \cite{bucher2018} for the maximum likelihood estimator of the parameters of a Fr\'{e}chet distribution  based on maxima of sliding or disjoint blocks, respectively, of a stationary time series with marginal distribution in the maximum domain of attraction of this Fr\'{e}chet distribution.

\cite{drees2010} provided a general framework to analyze the asymptotic behavior of statistics which are based on averages of functionals of disjoint blocks from an absolutely regular time series. Sufficient conditions for convergence of the empirical process of so-called cluster functionals established there proved to be a powerful tool for establishing asymptotic normality of a wide range of estimators; see, e.g., \cite{drees2015boot}, \cite{davis2018}, and \cite{knezevic2020}. Unfortunately, the setting considered by \cite{drees2010} is too restrictive to accommodate empirical processes based on sliding blocks.

The first aim of the present paper is thus to establish results on the convergence of empirical processes of the type
$$
 \bar Z_n(g) := \frac 1{\sqrt{p_n}b_n(g)} \sum_{j=1}^{n-s_n+1}\big( g(Y_{n,j})-Eg(Y_{n,j})\big), \quad g\in\GG.
$$
Here $Y_{n,j}$ is defined by \eqref{eq:Ydef} for some row-wise stationary triangular array $(X_{n,i})_{1\le i\le n, n\in\N}$,  $\GG$ is a set of functionals defined on vectors of arbitrary length that vanish if applied to a null vector and $\sqrt{p_n}b_n(g)$ is a normalizing sequence which will be introduced in Section~\ref{section:sliding}. We are mainly interested in the case when $X_{n,i}$ are suitably standardized extremes. In particular, we will assume $P\{\exists g\in\GG: g(Y_{n,1})\ne 0\}\to 0$. It is worth mentioning, though, that our general results can be applied to other statistics of rare events (cf.\ \cite{drees2010}, Ex.\ 3.5).

The second aim is to compare the performance of estimators derived from $\bar Z_n(g)$ with their analogs based on disjoint blocks. To this end, we will prove convergence of certain empirical processes in an abstract unifying framework which encompasses both the aforementioned setting to deal with sliding blocks processes $\bar Z_n$ and the setting discussed by \cite{drees2010}. This way one may derive the asymptotic normality of functionals of sliding resp.\ disjoint blocks under similar conditions, and the expressions obtained for their asymptotic variances become comparable. It will be shown that indeed, under weak conditions, the asymptotic variance of an estimator using sliding blocks statistics is never greater than the asymptotic variance of its counterpart based on disjoint blocks.

Sometimes block based extreme value statistics are motivated by the interpretation that all large values in such a block form a cluster of extremes. In another interpretation, all large values which are not separated in time by a certain number of smaller values form a cluster. This leads to so-called runs estimators, the best-known example of which is the estimator of the extremal index, proposed by \cite{hsing1993}. Such runs estimators can be considered as a special type of sliding blocks estimators and can thus be analyzed with the techniques developed in this paper under comparable conditions as estimators based on disjoint blocks.  It turns out that both types of estimators of the extremal index have the same asymptotic variance. While the asymptotic normality of both estimators has already been proved by \cite{weissman1998}, the equality of their asymptotic variances has been overlooked, because the variances were expressed differently. In addition, we establish the asymptotic normality of the direct sliding blocks analog to the disjoint blocks estimator. Under mild conditions, this estimator has the same asymptotic variance, too. This application demonstrates that, by analyzing different estimators of the same parameter in a unifying framework, one may gain new insights.

The paper is organized as follows. In Section \ref{section:sliding}, we first establish sufficient conditions for the convergence of empirical processes of sliding blocks statistics. Table \ref{tab:overview} provides an overview of several sequences of real and integer numbers arising in this context.  In Subsection \ref{section:sliidng.vs.disjoint}, the asymptotic variances of estimators using sliding and disjoint blocks, respectively, are compared. In Section \ref{section:extremal.index}, the general theory is applied to three estimators of the extremal index. Process convergence in the general abstract setting is presented in Appendix \ref{section:abstract}, while all proofs are collected in Appendix \ref{section:proofs}. Refinements to some of the results of this paper and detailed sufficient conditions for the asymptotic normality of statistics considered in Subsection \ref{section:sliidng.vs.disjoint} are presented in a Supplement.

Throughout the paper, $(E,\|\cdot \|)$ denotes a complete normed vector space and $E_{\cup}:=\bigcup_{n\in\mathbb{N}}E^{n}$ the set of vectors of arbitrary length with $E$-valued components. $\mathbb{N}$ denotes the natural numbers excluding $0$. For any doubly indexed sequence $Q_{n,i}$, $1\le i\le m_n$, of random variables that are identically distributed, $Q_n$ denotes a generic random variable with the same distribution as $Q_{n,1}$. Outer probabilities are denoted by $P^*$, outer expectations by $E^*$. Weak convergence is indicated by $\stackrel{w}{\to}$, while $\stackrel{P}{\to}$ denotes convergence in probability and $\stackrel{P^*}{\to}$ convergence in outer probability. The positive part of any $x\in\R$ is denoted by $x^+:=\max(x,0)$.

\section{Empirical processes of sliding blocks statistics}
\label{section:sliding}

Throughout this section we assume that $(X_{n,i})_{1\leq i\leq n, n\in\mathbb{N}}$ is a triangular array of row-wise stationary $E$-valued random variables.
First we establish conditions under which  an empirical process of sliding blocks statistics of the type
\begin{equation}  \label{eq:defbarZn}
\bar{Z}_n(g):=\frac{1}{\sqrt{p_n}b_n(g)}\sum_{j=1}^{n-s_n+1}\left(g(Y_{n,j})-E g(Y_{n,j})\right), \qquad g\in\mathcal{G},
\end{equation}
converges to a Gaussian process in the space $\ell^\infty(\GG)$ of bounded functions on $\GG$, endowed with the supremum norm. The normalizing sequence $\sqrt{p_n}b_n(g)\to \infty$ is discussed below.


To this end, we will apply the general abstract results presented in Appendix \ref{section:abstract} to
\begin{equation} \label{eq:Vnidefsliding}
V_{n,i}(g):= \frac{1}{b_n(g)}\sum_{j=1}^{r_n} g(Y_{n,(i-1)r_n+j})
\end{equation}
where $r_n$ denotes a sequence that grows faster than $s_n$ but slower than $n$. Furthermore, $r_n$ is chosen such that it is unlikely to have any extreme value in a sequence of $r_n$ consecutive observations. More precisely, we assume
\begin{equation} \label{eq:pndef}
  p_n:=P\{\exists g\in\mathcal{G} : V_{n}(g)\neq 0\}\to 0
\end{equation}
as $n\to\infty$, where $V_n$ has the same distribution as any $V_{n,i}$. The set $\{\exists g\in\mathcal{G} : V_{n}(g)\neq 0\}$ is measurable under the following condition:
\begin{itemize}
	\item[\bf(D0)] The processes $V_{n}$, $n\in\N$, are separable.
\end{itemize}
Condition (D0) helps to avoid measurability problems;  in particular, it is fulfilled if $\mathcal{G}$ is finite.
Note that $\bar{Z}_n$ can be approximated by
%
%
\begin{align}
\label{eq:zng.in.sliding.block}
Z_n(g)&:=\frac{1}{\sqrt{p_n}b_n(g)}\sum_{j=1}^{m_nr_n}\left(g(Y_{n,j})-E g(Y_{n,j})\right)\\
&=\frac{1}{\sqrt{p_n}}\sum_{i=1}^{m_n}\left(V_{n,i}(g)-E V_{n,i}(g)\right), \qquad g\in\mathcal{G},
\end{align}
with $m_n:=\floor{(n-s_n+1)/r_n}$. We will see below that under suitable conditions the last $n-s_n+1-m_nr_n<r_n$ summands in definition \eqref{eq:defbarZn} of $\bar Z_n$ are asymptotically negligible.

We will prove process convergence using the well-known ``big blocks, small blocks'' technique where each $Y_{n,j}$ takes over the role of a single observation and $r_n$ is the length of the big blocks. In addition, we need to choose the length $l_n$ of the smaller blocks which must not be smaller than $s_n$, so that $Y_{n,j}$ and $Y_{n,j+l_n}$ do not overlap.
Moreover, we assume that the dependence between observations separated in time by $l_n-s_n$ vanishes asymptotically. The strength of dependence will be measured by the mixing coefficients
\begin{equation} \label{eq:betaXdef}
 \beta_{n,k}^{X} := \sup_{1\le l\le n-k-1} E \Big[
\mathop{\text{sup}}_{B\in\BB_{n,l+k+1}^n} | P(B|\BB_{n,1}^l)-P(B)|\Big]
\end{equation}
where $\BB_{n,i}^j$ denotes the
$\sigma$-field generated by $(X_{n,l})_{i\le l\le j}$.
To summarize, we require the following conditions on the observational scheme, the different sequences and the function class:
\begin{itemize}
	\item[\bf(A1)] $(X_{n,i})_{1\leq i\leq n}$ is stationary for all $n\in\mathbb{N}$.
	
	\item[\bf(A2)] The sequences $l_n,r_n,s_n\in\N$, $p_n$ defined in \eqref{eq:pndef}, and $b_n(g)>0$, $g\in\GG$, satisfy $s_n\le l_n=\ord(r_n)$, $r_n=\ord(n)$, $p_n\to 0$ and $r_n=\ord\big(\sqrt{p_n} \inf_{g\in\GG}b_n(g)\big)$.
	\item[\bf(MX)] $m_n\beta_{n,l_n-s_n}^{X}\to 0$ for $m_n:=\floor{(n-s_n+1)/r_n}$.
\end{itemize}
An overview of the sequences and their interpretations can be found in Table \ref{tab:overview}.
Finally, to ensure the convergence of the finite dimensional marginal distributions (fidis) of $\bar Z_n$, we assume
\begin{itemize}
	\item[\bf(C)] There exists a function $c:\GG^2\to\R$ such that
\begin{equation}
	\frac{m_n}{p_n}Cov\left(V_{n}(g),V_{n}(h)\right)\rightarrow c(g,h), \qquad \forall\, g,h\in\mathcal{G}.
	\end{equation}
\end{itemize}	
Our first result deals with the convergence of the fidis if $\GG$ is uniformly bounded.
\begin{theorem}
	\label{cor: g besch bedingung sliding blocks}
	Suppose $g_{\max}=\sup_{g\in\mathcal{G}}|g|$ is bounded and measurable and the conditions
	(A1), (A2), (D0) and (MX) are met.
	Moreover, assume
	\begin{equation}
	\label{eq: bed g besch anzahl ungleich 0}
	E\bigg[\bigg(\sum_{j=1}^{r_n} \mathds{1}_{\left\{g(Y_{n,j}) \neq 0\right\}}\bigg)^2 \bigg] =\Ord\bigg(\frac{p_nb_n^2(g)}{m_n}\bigg), \qquad\forall g\in\mathcal{G}.
	\end{equation}
	Then
	\begin{equation}
	\label{eq:sliding fidi con wie ursprung proz}
	  \sup_{g\in \GG}|Z_n(g)-\bar{Z}_n(g)|\xrightarrow{P^*} 0.
	\end{equation}
	If, in addition, (C) is fulfilled, then the fidis of each of the empirical processes $(Z_n(g))_{g\in\mathcal{G}}$ and $(\bar{Z}_n(g))_{g\in\mathcal{G}}$ converge weakly to the fidis of a Gaussian process $(Z(g))_{g\in\mathcal{G}}$ with covariance function $c$. 	
\end{theorem}

The following criterion is often useful to verify condition \eqref{eq: bed g besch anzahl ungleich 0}:
\begin{itemize}
	\item[\bf(S)] For all $g\in\mathcal{G}$ and $n\in\mathbb{N}$
	\begin{equation}
	       \sum_{k=1}^{r_n}  P\left\{g(Y_{n,1})\neq 0, g(Y_{n,k})\neq 0\right\} = \Ord\bigg( \frac{p_n b_n^2(g)}n\bigg).
	\end{equation}
\end{itemize}

\begin{lemma}
	\label{lemma:pratts lemma bedingung allgemeines setting}
	If condition (S) is satisfied, then \eqref{eq: bed g besch anzahl ungleich 0} holds.
\end{lemma}

For instance, in Section 3 we consider the bounded functions
\begin{align}
g_1(x_1,\ldots,x_s) :=\mathds{1}_{\{\max_{1\le i\le s} x_i>1\}},\qquad
g_2(x_1,\ldots,x_s) :=\mathds{1}_{\{x_1>1\}}
\end{align}
to analyze the sliding blocks estimator of the extremal index. Here appropriate  normalizing sequences are $b_n(g_1)= \sqrt{m_n}s_n $ and $b_n(g_2)= \sqrt{m_n}$. Note that already in this rather simple example, the normalizing sequences converge at a different rate for different functions. Indeed, it is somewhat archetypical that the event $g(Y_{n,1})\neq 0$ either depends on all observations of the block $Y_{n,1}$ (as for $g=g_1$), or it only depends on a single fixed observation $X_{n,i}$ (as for $g=g_2$); usually the normalizing factor $b_n(g)$ is larger by the factor $s_n$ in the former case.

To ensure asymptotic equicontinuity or tightness of the processes $(Z_n(g))_{g\in\mathcal{G}}$ and $(\bar{Z}_n(g))_{g\in\mathcal{G}}$, and thus  process convergence if the conditions of Theorem \ref{cor: g besch bedingung sliding blocks} are fulfilled, we need the following additional conditions.
\begin{itemize}
	\item[\bf(D1)] There exists a semi-metric $\rho$ on $\mathcal{G}$ such that $\mathcal{G}$ is totally bounded (i.e.\ for all $\epsilon>0$, it can be covered by finitely many balls with radius $\epsilon$ w.r.t.\  $\rho$) and
	\begin{equation}
	\lim_{\delta\downarrow 0} \limsup_{n\to\infty} \sup_{g,h\in\mathcal{G},\rho(g,h)<\delta} \frac{m_n}{p_n } E\big[(V_n(g)-V_n(h))^2\big] = 0.
	\end{equation}

	\item[\bf(D2)]
	\begin{equation}
	\lim_{\delta\downarrow 0}\limsup_{n\to\infty}\int_{0}^{\delta} \sqrt{\log N_{[\cdot]}(\epsilon,\mathcal{G},L_2^n) } \, d\epsilon =0,
	\end{equation}
	where $N_{[\cdot]}(\epsilon,\mathcal{G},L_2^n)$ denotes the $\epsilon$-bracketing number of $\GG$ w.r.t.\ $L_2^n$,  i.e.\ the smallest number $N_\epsilon$ such that for each $n\in\mathbb{N}$ there exists a partition $(\mathcal{G}_{n,k}^{\epsilon})_{1\leq k\leq N_{\epsilon}}$ of $\mathcal{G}$ satisfying
	\begin{equation}
	\frac{m_n}{p_n } E^*\Big[\sup_{g,h\in\mathcal{G}_{n,k}^{\epsilon}} (V_n(g)-V_n(h))^2\Big]\leq \epsilon^2,  \qquad\forall 1\leq k\leq N_{\epsilon}.
	\end{equation}
	
	
	\item[\bf(D3)] Denote by $N(\epsilon,\mathcal{G},d_n)$ the $\epsilon$-covering number of $\GG$ w.r.t.\ the random semi-metric
	\begin{equation}
	d_n(g,h)=\bigg(\frac{1}{p_n}\sum_{i=1}^{m_n}(V_{n,i}^*(g)-V_{n,i}^*(h))^2\bigg)^{1/2}
	\end{equation}
	 with  $V_{n,i}^*$, $1\leq i\leq m_n$,  independent copies of $V_{n,1}$, i.e.\ $N(\epsilon,\mathcal{G},d_n)$ is the smallest number of balls with respect to $d_n$ with radius $\epsilon$ which is needed to cover $\mathcal{G}$. We assume
	\begin{equation}
	\lim_{\delta\downarrow 0} \limsup_{n\to\infty} P^*\bigg\{\int_{0}^{\delta} \sqrt{\log(N(\epsilon,\mathcal{G},d_n))}d\epsilon>\tau\bigg\}=0,\qquad\forall \tau>0.
	\end{equation}
\end{itemize}
Roughly speaking, condition (D1) ensures the continuity of the process w.r.t.\ $\rho$ while (D2) and (D3) ensure that the parameter set $\GG$ is not too complex. In particular, condition (D3) is satisfied if $\mathcal{G}$ is a VC-class (cf.\ \cite{drees2010}, Remark 2.11).


\begin{theorem}
	\label{cor: g besch bedingung sliding blocks process}
	Suppose the conditions of Theorem \ref{cor: g besch bedingung sliding blocks} are satisfied.
	If, in addition, one of the following sets of conditions
	\begin{itemize}
		\item[(i)] (D1) and (D2), or
        \item[(ii)]   (D1) and (D3)	
	\end{itemize}
	is fulfilled, then each of the empirical processes $(Z_n(g))_{g\in\mathcal{G}}$ and $(\bar{Z}_n(g))_{g\in\mathcal{G}}$ converge weakly to a Gaussian process with covariance function $c$.	
\end{theorem}

\begin{table}[t]
	\begin{tabular}{l|l|l|l|l|l}
		& \textit{interpretation} & \textit{$\to$} & \textit{main constraints} & \textit{typ.\ behavior} &  \textit{first use} \\\hline
		$n$ & number of observations & $\infty$ & & & p.\pageref{page.n} \\ \hline
		$s_n$ & length of sliding blocks &  &  & $s_n\to\infty$ & \eqref{eq:Ydef} \\\hline
		$r_n$ & length of big block & $\infty$ & in Sect.\ 2: \parbox{3.5cm}{$r_n=\ord(n)$\\ $r_n=\ord\big(\sqrt{p_n} \inf_{g\in\GG}b_n(g)\big)$}  & & \eqref{eq:Vnidefsliding} \\
		& & & in Sect.\ 3: \parbox{3.5cm}{$r_nv_n\to 0$, $r_n=\ord\big(\sqrt{nv_n}\big)$}   & & \\\hline
		$l_n$ & length of small block & $\infty$ & $s_n\leq l_n =\ord(r_n)$ & & (A2) \\\hline
		$m_n$ & number of big blocks & $\infty$ & $m_n\asymp n/r_n$ & & \eqref{eq:zng.in.sliding.block} \\\hline
		$u_n$ & threshold for $X$ to be large & $\infty$ & & & p.\pageref{page.n} \\\hline
		$v_n$ & $P\{X_{n,1}\neq 0\}$& 0 & \parbox{4cm}{$nv_n\to \infty$ } & & p.\pageref{page.v}\\ \hline
		$p_n$ & $P\{\exists 1\leq i\leq r_n: \,X_{n,i}\neq 0\}$ & $ 0$ & $r_n=\ord\big(\sqrt{p_n} \inf_{g\in\GG}b_n(g)\big)$ & $p_n\asymp r_nv_n$ & \eqref{eq:pndef} \\\hline
		$b_n(g)$ & normalizing constant & $\infty$ & $\sqrt{p_n}b_n(g)\to\infty$ & \parbox{2.2cm}{ $b_n(g)\asymp\sqrt{m_n}$ or $b_n(g)\asymp\sqrt{m_n}s_n$} & \eqref{eq:defbarZn}  \\\hline
		$a_n$ & normalization in Section \ref{section:sliidng.vs.disjoint} &  &  & $a_n\asymp 1$ & p.\pageref{page.v} \\\hline
	\end{tabular}
	\caption{Overview of sequences occurring in Sections \ref{section:sliding} and \ref{section:sliidng.vs.disjoint}.}
	\label{tab:overview}
\end{table}

So far, we have only discussed the case of bounded functions $g$. This assumption can be dropped if the moment condition \eqref{eq: bed g besch anzahl ungleich 0} is strengthened.

\begin{theorem}
	\label{cor:g unbounded sliding block con}
   \begin{itemize}
	 \item[(i)] Suppose all conditions of Theorem \ref{cor: g besch bedingung sliding blocks} except for the boundedness of $g_{\max}$ and \eqref{eq: bed g besch anzahl ungleich 0} are met. In addition, we assume $m_n l_nP\{V_n(|g|)\ne 0\}=\ord(r_nb_n^2(g)p_n)$ for all $g\in\GG$ and
	\begin{equation}
	\label{eq: bed g messbar bedingung sliding blocks}
	E\bigg[\bigg( \sum_{i=1}^{r_n}|g(Y_{n,i})|\bigg)^{2+\delta}\bigg]=\Ord\bigg(\frac{p_nb_n^2(g)}{m_n}\bigg),\quad \forall\, g\in\GG,
	\end{equation}
	for some $\delta>0$. Then the fidis of $(Z_n(g))_{g\in\mathcal{G}}$ and of $(\bar{Z}_n(g))_{g\in\mathcal{G}}$ converge to the fidis of the Gaussian process $(Z(g))_{g\in\mathcal{G}}$ defined in Theorem \ref{cor: g besch bedingung sliding blocks}.
    \item[(ii)] If, in addition, $b_n(g)=b_n$ is the same for all $g\in\GG$, \eqref{eq: bed g messbar bedingung sliding blocks} holds for $g=g_{\max}$ and the conditions (i) or (ii) of Theorem \ref{cor: g besch bedingung sliding blocks process} are fulfilled, then the processes $(Z_n(g))_{g\in\mathcal{G}}$ and $(\bar{Z}_n(g))_{g\in\mathcal{G}}$ converge weakly to $(Z(g))_{g\in\mathcal{G}}$ uniformly.
  \end{itemize}
\end{theorem}
Note that usually $P\{V_n(|g|)\ne 0\}=O(p_n)$; in particular this holds true if $g$ has fixed sign. Then the condition $m_n l_nP\{V_n(|g|)\ne 0\}=\ord(r_nb_n^2(g)p_n)$ is fulfilled for the typical behavior of the sequences outlined in Table \ref{tab:overview}.
As mentioned above, usually it suffices to consider just two different normalizing sequences, say $b_{n,1}$ and $b_{n,2}$. In this case, one may apply Theorem~\ref{cor:g unbounded sliding block con} separately to $(Z_n(g))_{g\in\GG_i}$ for $i\in\{1,2\}$ with $\GG_i:=\{g\in\GG | b_n(g)=b_{n,i},\;\forall\, n\in\N\}$ to conclude that both processes are asymptotically tight. This in turn implies the asymptotic tightness of $(Z_n(g))_{g\in\GG}$ and thus, in view of part (i), its convergence to $(Z(g))_{g\in\mathcal{G}}$. Hence, in fact the extra condition on $b_n$ in part (ii) does not further restrict the setting in the vast majority of applications.

\subsection{Sliding vs.\ disjoint blocks statistics}
\label{section:sliidng.vs.disjoint}

The previous section was devoted to  general limit theorems for sliding blocks statistics. In this section, we want to compare the asymptotic variance of a sliding blocks statistic for a single functional $g$ with that of the corresponding disjoint blocks statistic. Here we use a different parametrization of the normalizing constants, partly because the probability $p_n$ used in the normalization above refers to the whole process and seems inappropriate in the present context, partly to facilitate the comparison of the asymptotic variances. More precisely, we consider the sliding blocks statistic and its disjoint blocks analog
\begin{align}
T_{n}^s(g) & :=\frac{1}{n v_n s_n a_n}\sum_{i=1}^{n-s_n+1}g(Y_{n,i}) \label{eq:vergl sliding stat}\\
T_{n}^d(g) & :=\frac{1}{n v_n a_n}\sum_{i=1}^{\lfloor n/s_n\rfloor}g(Y_{n,(i-1)s_n+1}), \label{eq:vergl disjoin stat}
\end{align}
with $v_n:=P(X_{n,1}\neq 0)\to 0$.\label{page.v} We assume that $a_n$ is chosen such that $E(T_{n}^s(g))$ converges in $\R$, i.e.\ that there exists some $\xi\in\mathbb{R}$ such that
\begin{equation} \label{eq:slidingexpectconv}
E\left[ T_n^s(g)\right]= \frac{1}{s_nv_na_n}E\left[g(Y_{n})\right]\frac{n-s_n+1}{n}\rightarrow \xi.
\end{equation}
Then also $E(T_{n}^d(g))$ tends to $\xi$. Moreover, the difference between both expectations is asymptotically negligible if
\begin{equation}
\label{eq:disjoint vs sliding expectation}
\Big|E\left[ T_n^d(g)-T_n^s(g)\right]\Big|=\frac{1}{s_nv_na_n}|E\left[g(Y_{n})\right]|\cdot \Big|\frac{s_n}{n}\Big\lfloor \frac{n}{s_n}\Big\rfloor -\frac{n-s_n+1}{n}\Big| =\Ord(s_n/n)
\end{equation}
is of smaller order than $(nv_n)^{-1/2}$ (cf.\ \eqref{eq:vergle sliding stat con}, \eqref{eq:vergl disjoint stat con}), which in particular holds under the basic condition $s_nv_n\to0$. In that case, $T_n^s(g)$ will be a more efficient estimator than $T_n^d(g)$ if its asymptotic variance is smaller.


Applying Theorem \ref{cor: g besch bedingung sliding blocks} with $b_n(g)=\sqrt{nv_n/p_n}a_ns_n$, under suitable conditions including
the convergence
\begin{equation} \label{eq:csdef}
  c^{(s)}=\lim_{n\to\infty}\frac{1}{r_nv_ns_n^2a_n^2}Var\bigg(\sum_{i=1}^{r_n}g(Y_{n,i})\bigg)\in (0,\infty),
\end{equation}
 one can prove the asymptotic normality of the sliding blocks statistics
\begin{equation}
\label{eq:vergle sliding stat con}
\sqrt{nv_n}\big(T_n^{s}(g)-E[T_n^{s}(g)]\big)\stackrel{w}{\to} \mathcal{N}(0,c^{(s)}).
\end{equation}
To establish an analogous result for the statistic based on disjoint blocks, one applies Theorem \ref{satz: emp prozess fidi} to $V_{n,i}(g)=\sqrt{p_n/(nv_n a_n^2)}\sum_{j=1}^{r_n/s_n}g(Y_{n,(j-1)s_n+(i-1)r_n+1})$, $1\le i\le m_n$.
Recall that the sequence $r_n$ is only needed in the proofs which use the ``big blocks, small blocks'' technique, i.e.\ it has no operational meaning, but it must be chosen such that the conditions of Theorem \ref{cor: g besch bedingung sliding blocks} resp.\ Theorem \ref{satz: emp prozess fidi} are met. For example, suppose that for a given sequence $(s_n)_{n\in\N}$, $(r_n)_{n\in\N}$ is a sequence such that the assumptions of Theorem \ref{cor: g besch bedingung sliding blocks} are satisfied. Let $r_n^*:=\floor{r_n/s_n}s_n\sim r_n$, so that $l_n=\ord(r_n^*)$, $r_n^*=\ord(n)$ and $m_n^*:=\floor{(n-s_n+1)/r_n^*}\sim m_n$. Moreover, the proof of Theorem \ref{cor: g besch bedingung sliding blocks} (cf.\ \eqref{eq:Deltan2ndmombound}) shows that for
$$ V_{n,1}^*(g) := \frac{1}{b_n(g)} \sum_{j=1}^{r_n^*} g(Y_{n,j}), \quad\text{and} \quad p_n^* := P\{\exists g\in\GG: V_{n,1}^*(g)\ne 0\},
$$
one has
\begin{align}
  E\big[(V_{n,1}^*(g)-V_{n,1}(g))^2\big] & = E\bigg[\Big(\frac{1}{b_n(g)}\sum_{j=r_n^*+1}^{r_n} g(Y_{n,j})\Big)^2\bigg] = \ord\Big(\frac{p_n}{m_n}\Big),\\
  |p_n^*-p_n| & \le s_n v_n.
\end{align}
Hence, if $p_n\asymp r_nv_n$ (which holds true for all known examples), $p_n^*\sim p_n$ and the conditions of Theorem 2.1 are still fulfilled if one replaces $r_n$ with $r_n^*$. One may argue similarly in the setting of Theorem 2.4.

We may thus assume w.l.o.g.\ that $r_n$ is a multiple of $s_n$, where the multiplicity depends on $n$. Note that $r_n/s_n$ must tend to $\infty$ if Theorem 2.1 shall be applied.
We then obtain
\begin{equation}
\label{eq:vergl disjoint stat con}
\sqrt{nv_n}\left(T_n^{d}(g)-E[T_n^{d}(g)]\right)\stackrel{w}{\to} \mathcal{N}(0,c^{(d)}),
\end{equation}
with
\begin{equation}  \label{eq:cddef}
c^{(d)}=\lim_{n\to\infty} \frac{1}{r_nv_na_n^2} Var\bigg(\sum_{i=1}^{r_n/s_n}g(Y_{n,is_n+1})\bigg).
\end{equation}
See the Supplement for details about the conditions under which \eqref{eq:vergle sliding stat con} and \eqref{eq:vergl disjoint stat con} hold. Alternatively, one could prove the asymptotic normality of $T_n^{d}(g)$ using Theorem 2.3 of \cite{drees2010} with $r_n$ replaced by $s_n$, but the above representation of the asymptotic variance $c^{(d)}$ simplifies the comparison with $c^{(s)}$. The following theorem shows that the asymptotic variance of the sliding blocks statistic is never greater than that of the disjoint blocks statistic.
\begin{theorem}
	\label{lemma:slidng vs disjoint zaehler}
	If conditions (A1), \eqref{eq:csdef} and \eqref{eq:cddef} hold, and  $r_n/s_n\in\mathbb{N}$ for all $n\in\mathbb{N}$, then $c^{(s)}\leq c^{(d)}$.
\end{theorem}
Indeed, one can even prove a multivariate version of this theorem: under suitable conditions the asymptotic covariance matrix of a vector of sliding blocks statistics $\big(T_n^{s}(g_i))_{1\le i\le I}$ is smaller w.r.t.\ the Loewner order than the corresponding matrix of the disjoint blocks statistics (see Supplement).

Usually, the probability $v_n$ that a single observation $X_{n,1}$ does not vanish is unknown, whereas the normalizing constant $a_n$ may depend on $g$, but not on the unknown distribution of $X_{n,1}$. In what follows, we thus analyze versions of our statistics where $v_n$ is replaced with a simple empirical estimator. This results in the estimators
\begin{align}
\label{eq:vergl sliding est}
\tilde T_n^s(g) & := \frac{nv_n T_n^s(g)}{\sum_{i=1}^{n-s_n+1}\mathds{1}_{\{X_{n,i}\neq 0\}}} = \frac{\frac{1}{s_na_n}\sum_{i=1}^{n-s_n+1}g(Y_{n,i})}{\sum_{i=1}^{n-s_n+1}\mathds{1}_{\{X_{n,i}\neq 0\}}},\\
\tilde T_n^d(g) & := \frac{nv_n T_n^d(g)}{\sum_{i=1}^{n-s_n+1}\mathds{1}_{\{X_{n,i}\neq 0\}}} =\frac{\frac{1}{a_n}\sum_{i=1}^{\lfloor n/s_n\rfloor}g(Y_{n,(i-1)s_n+1})}{\sum_{i=1}^{n-s_n+1}\mathds{1}_{\{X_{n,i}\neq 0\}}}
\end{align}
of $\xi$.
In order to prove convergence of these estimators, one needs the joint convergence of the numerator and denominator. This can again be concluded from Theorem \ref{cor: g besch bedingung sliding blocks} or Theorem \ref{satz: emp prozess fidi}, respectively, now applied with $\GG=\{g,h\}$ and $h(x_1,...,x_{s})=\mathds{1}_{\{x_{1}\neq 0\}}$.
Similarly as before, one obtains
\begin{equation}
\label{eq:vergl est con}
\sqrt{nv_n}\begin{pmatrix}
       T_n^\sharp(g)-E[T_n^\sharp(g)]\\
\frac{1}{nv_n}\sum_{i=1}^{n-s_n+1}\big(\mathds{1}_{\{X_{n,i}\neq 0\}}-v_n\big) \\
\end{pmatrix}
\xrightarrow{w}  \mathcal{N}_2\left(0,\begin{pmatrix}	c^{(\sharp)}  & c^{(\sharp,v)} \\
c^{(\sharp,v)}  & c^{(v)}
\end{pmatrix}\right).
\end{equation}
where $\sharp$ stands either for $d$ or $s$
and
\begin{align}
c^{(s,v)}&:=\lim_{n\to\infty}\frac{1}{r_nv_ns_na_n}Cov\bigg(\sum_{i=1}^{r_n}g(Y_{n,i}),\sum_{i=1}^{r_n}\mathds{1}_{\{X_{n,i}\neq 0\}}\bigg),\\
c^{(d,v)}&:=\lim_{n\to\infty} \frac{1}{r_nv_na_n}Cov\bigg(\sum_{j=1}^{r_n/s_n}g(Y_{n,(j-1)s_n+1}),\sum_{i=1}^{r_n}\mathds{1}_{\{X_{n,i}\neq 0\}}\bigg),\\
c^{(v)}&:=\lim_{n\to\infty}\frac{1}{r_nv_n} E\bigg[\bigg(\sum_{i=1}^{r_n} \mathds{1}_{\{X_{n,i}\neq 0\}}\bigg)^2  \bigg].
\end{align}
Note that the same result holds if $\sum_{i=1}^{n-s_n+1}$ is replaced with $\sum_{i=1}^n$ (cf. \eqref{eq:sliding fidi con wie ursprung proz}).

By some standard continuous mapping argument (see Supplement), one may conclude
\begin{align}
\label{eq:vergl disjoint est con frac}
\sqrt{nv_n}\big( \tilde T_n^{\sharp} - \xi\big)\stackrel{w}{\to} \mathcal{N}(0,\tilde{c}^{(\sharp)})
\end{align}
with $\tilde{c}^{(\sharp)}:=c^{(\sharp)}+\xi^2 c^{(v)}-2\xi c^{(\sharp,v)}$, provided the bias of the estimator is negligible, that is
 $E[g(Y_{n})]/s_nv_na_n-\xi=\ord\big((nv_n)^{-1/2}\big)$.

It turns out that under rather mild conditions again the asymptotic variance of the estimator using sliding blocks is not greater than that of the disjoint blocks estimator, if the function $g$ has constant sign.
\begin{theorem}
	\label{satz:sliding vs disjoint ganzer schaetzer}
	Suppose the conditions of Theorem  \ref{lemma:slidng vs disjoint zaehler} are satisfied, \eqref{eq:slidingexpectconv} holds,
	the function $g$ is bounded and does not change its sign, $s_n=\ord(r_n a_n)$ and $s_nv_n\to 0$.
	If, in addition, there exists a sequence $k_n=\ord(r_na_n)$ of natural numbers such that the $\beta$-mixing coefficients defined in \eqref{eq:betaXdef} satisfy $\sum_{i=k_n}^{r_n}\beta_{n,i}^X=\ord(r_nv_na_n)$, then $\tilde{c}^{(s)}\leq \tilde{c}^{(d)}$.	
\end{theorem}
In fact, it can be shown that $\tilde{c}^{(d)}-\tilde{c}^{(s)}={c}^{(d)}-{c}^{(s)}$. In the most common case that the mixing coefficients decrease exponentially fast and $\log n=\ord(r_na_n)$, the sequence $k_n=\floor{c\log n}$ with sufficiently large constant $c>0$ fulfills the conditions of Theorem \ref{satz:sliding vs disjoint ganzer schaetzer}.

\section{Estimating the extremal index}
\label{section:extremal.index}
In this section we  apply the general theory presented in Section \ref{section:sliding} and Appendix  \ref{section:abstract} to analyze the asymptotic behavior  of three estimators for the extremal index of a real-valued stationary time series $(X_t)_{t\in\Z}$. If for all thresholds $u_n(\tau)$ such that $nP\{X_0>u_n(\tau)\}\to\tau$ for some $\tau>0$ one has
\begin{equation}
\lim_{n\to\infty}P\Big\{\max_{1\leq i\leq n}X_i\leq u_n(\tau)\Big\} = e^{-\theta \tau},
\end{equation}
then $\theta$ is said to be the \textit{extremal index} of the time series (\cite{leadbetter1983}). The extremal index always lies in $[0,1]$. In what follows, we exclude the degenerate case $\theta=0$ and assume $\theta>0$.

The estimation of this extremal index has been much discussed  in the literature, see e.g. \cite{smith1994}, \cite{ferro2003}, \cite{suveges2007}, \cite{robert2009}, \cite{berghaus2018}, among others. We examine two of the most popular estimators, the blocks and the runs estimator, and a variant of the former. Throughout this section, we use the notation $M_{i,j}:=\max(X_i,...,X_j)$.

If the extremal index exists then, under weak additional conditions,
\begin{equation}
\label{eq:maximum zu einzelueberschreitung}
\frac{P\{M_{1,k_n}>u_n\}}{k_nP\{X_1>u_n\}}\rightarrow \theta
\end{equation}
for sequences $k_n\to\infty$ and $u_n$ such that $k_nP\{X_1>u_n\}\to 0$. In particular, this holds if $\beta^X_{n,l_n}/(k_nv_n)\to 0$ for some $l_n=\ord(k_n)$ (cf.\ \cite{leadbetter1983}, Theorem 3.4).
If one replaces the unknown probabilities by empirical ones, using disjoint blocks to estimate the numerator for $k_n=s_n$, one arrives at the following  estimator proposed by \cite{hsing1991}:
\begin{equation}
\label{def:extremal index disjoint}
\hat{\theta}_n^d:= \frac{\sum_{i=1}^{\lfloor n/s_n\rfloor}\mathds{1}_{\{M_{(i-1)s_n+1,is_n}> u_n\}}}{\sum_{i=1}^{n-s_n+1}\mathds{1}_{\{X_i>u_n\}}}.
\end{equation}
He proved asymptotic normality of this blocks  estimator under some tailor-made conditions.
As suggested in Section 10.3.4 of \cite{beirlant2004}, alternatively one may use sliding blocks, which leads to
\begin{equation}
\label{def:extremal index sliding}
\hat{\theta}_n^s:= \frac{\frac{1}{s_n}\sum_{i=1}^{n-s_n+1}\mathds{1}_{\{M_{i,i+s_n-1}> u_n\}}}{\sum_{i=1}^{n-s_n+1}\mathds{1}_{\{X_i>u_n\}}}.
\end{equation}

The so-called runs estimator of $\theta$ is based on the following characterization of the extremal index:
\begin{equation}
\label{eq:runs char extremal}
P(M_{2,k_n}\leq u_n|X_1>u_n)\rightarrow \theta,
\end{equation}
which was first proven by \cite{obrien1987} under suitable conditions.
Again, by replacing the unknown probabilities for $k_n=s_n$ by empirical counterparts, one arrives at
\begin{equation}
\label{def:extremal index runs}
\hat{\theta}_n^r:=\frac{\sum_{i=1}^{n-s_n+1}\mathds{1}_{\{X_i>u_n,M_{i+1,i+s_n-1}\leq u_n\}}}{\sum_{i=1}^{n-s_n+1}\mathds{1}_{\{X_i>u_n\}}}.
\end{equation}
This runs estimator was suggested by \cite{hsing1993}. Its asymptotic normality was first established in \cite{weissman1998} who also proved the asymptotic normality of $\hat\theta_n^d$ under somewhat simpler conditions than \cite{hsing1991}. For a very specific model, \cite{weissman1998} showed that the asymptotic variances of both estimators are the same, but they did not realize that this is indeed true under quite general structural assumptions, as we will show below.

To establish asymptotic normality of these estimators, we need the following conditions:
\begin{itemize}
	\item[\bf($\theta$1)] For $v_n:=P\{X_1>u_n\}\to 0$, one has $nv_n\to \infty$ and $s_n\to\infty$. In addition, there exists a sequence $(r_n)_{n\in\N}$ such that $s_n=\ord(r_n)$, $r_nv_n\to 0$, $r_n=\ord(\sqrt{nv_n})$ and  $(n/r_n)\beta^X_{n,s_n-1}\to 0$.
	
	\item[\bf($\theta$2)]
	$\displaystyle c:=\lim_{n\to\infty}\frac{1}{r_nv_n} E\bigg[\bigg(\sum_{j=1}^{r_n} \mathds{1}_{\{X_{j}>u_n\}}\bigg)^2  \bigg]$ exists in $[0,\infty)$.
	
	\item[\bf($\theta$P)] For  all $n\in\mathbb{N}$ and $k\in\mathbb{N} $ there exists $e_n(k)$ such that
	\begin{equation}
	e_n(k)\geq P(X_k>u_n|X_0>u_n)
	\end{equation}
	and $\lim_{n\to\infty}\sum_{k=1}^{r_n}e_n(k)=\sum_{k=1}^{\infty}\lim_{n\to\infty}e_n(k) <\infty$.
\end{itemize}

By Pratt's lemma (\cite{pratt1960}), condition ($\theta$P) enables us to exchange sums and limits in the calculation of variance and covariance. Moreover, under ($\theta$1) and ($\theta$P), both \eqref{eq:maximum zu einzelueberschreitung} and \eqref{eq:runs char extremal} hold for all $k_n\le r_n$ such that $k_n\to \infty$. This follows from Theorem 1 and Corollary 2 of \cite{segers2003} in combination with the aforementioned result on convergence~\eqref{eq:maximum zu einzelueberschreitung}.

The limit $c$ is the asymptotic variance of the estimator for $v_n=P\{X_i>u_n\}$. If ($\theta$P) holds and the positive part $(X_t^+)_{t\in\mathbb{Z}}$ of the time series  is regular varying, then $c$ can be represented in terms of its tail process $(W_t)_{t\in\mathbb{Z}}$ (see Supplement), i.e.\ ($\theta$2) holds with
\begin{align}
c
& = 1+ \lim_{n\to\infty} \sum_{k=1}^{r_n-1}\Big(1-\frac{k}{r_n}\Big)\big(P(X_{k}>u_n|X_0>u_n)+P(X_{0}>u_n|X_{-k}>u_n)\big)\\
&=1+2\sum_{k=1}^\infty P\{W_{k}>1\}\label{eq:example calculation c for reg var}.
\end{align}
Alternatively, one may use the representation $c=\sum_{k\in\mathbb{Z}} P\{W_{k}>1\}$.

In addition, we have to assume that convergence \eqref{eq:maximum zu einzelueberschreitung} for $k_n=s_n$ and convergence  \eqref{eq:runs char extremal}, respectively, is sufficiently fast to ensure that the bias of the block based estimators or runs estimators, respectively, is asymptotically negligible:
\begin{itemize}
	\item[\bf{(B$_b$)}]
	$ \displaystyle
	\frac{P\{M_{1,s_n}> u_n\}}{s_nv_n}  - \theta = \ord\big((nv_n)^{-1/2}\big).
	$
	\item[\bf{(B$_r$)}]
	$ \displaystyle
	P(M_{2,s_n}\leq u_n|X_1>u_n)  - \theta\ = \ord\big((nv_n)^{-1/2}\big).
	$
\end{itemize}

The following result shows that under our conditions all three estimator have the same limit distribution.
\begin{theorem}
	\label{thm:extremal index}
	If the conditions ($\theta$1), ($\theta$2) and ($\theta$P) are satisfied, then
	\begin{equation}
	\sqrt{nv_n}(\hat{\theta}_n^\sharp-\theta)\xrightarrow{w} \mathcal{N}(0, \theta(\theta c -1)),
	\end{equation}
	provided (B$_b$) holds when $\sharp$ stands for `d' or `s', and (B$_r$) holds when $\sharp$ stands for `r'.
\end{theorem}

In practice, usually the threshold $u_n$ is replaced with some data driven choice $\hat u_n$, like an intermediate order statistic of the observed time series. By the techniques developed in \cite{knezevic2020}, one may prove that these versions of the  estimators of the extremal index asymptotically behave the same, provided $\hat u_n/u_n\stackrel{P}{\to} 1$ and the time series $(X_t^+)_{t\in\mathbb{Z}}$ is regular varying. To this end, the results about the convergence of the fidis are not sufficient any more, but the full process convergence is needed. The precise results and their proofs are given in the Supplement.

\appendix

\section{Functional limit theorems in an abstract setting}
\label{section:abstract}
In this section we prove abstract limit theorems for empirical processes which imply both the limit theorems \ref{cor: g besch bedingung sliding blocks}, \ref{cor: g besch bedingung sliding blocks process} and \ref{cor:g unbounded sliding block con} for statistics of sliding blocks and the limit theorems established by \cite{drees2010}.
As in Section \ref{section:sliding} we consider a triangular array $(X_{n,i})_{1\leq i\leq n, n\in\mathbb{N}}$ of row-wise stationary $E$-valued random variables. Fix sequences $r_n=\ord(n)$ and $s_n=\ord(r_n)$ of natural numbers. In what follows,
 $V_{n,i}(g)$ are real-valued random variables that are measurable w.r.t.\ $(X_{n,(i-1)r_n+1},...,X_{n,ir_n+s_n-1})$, for all $1\leq i\leq m_n$ and $g\in\mathcal{G}$, which are assumed to form a stationary sequence of processes. We are interested in the weak convergence of
\begin{equation}
\label{def:z.emp.process}
Z_n(g):=\frac{1}{\sqrt{p_n}}\sum_{i=1}^{m_n}\left(V_{n,i}(g)-E[V_{n,i}(g)]\right), \qquad g\in\mathcal{G},
\end{equation}
where $m_n:=\lfloor (n-s_n+1)/r_n \rfloor$ and $p_n:=P\{\exists g\in\mathcal{G}: V_{n}(g)\neq 0 \}\to 0$ is assumed.

The choice $V_{n,i}(g)=m_n^{-1/2}\sum_{j=1}^{r_n}g(X_{n,(i-1)r_n+j})$ leads to the generalized tail array sums examined in Section 3 of \cite{drees2010}. Sums of more general statistics of disjoint blocks can be analyzed using $V_{n,i}(g)=\sum_{j=0}^{r_n/s_n-1} g(Y_{n,(i-1)r_n+js_n+1})/d_n(g)$ for suitable normalizing sequences $d_n(g)$ (assuming that $r_n$ is a multiple of $s_n$), while the choice \eqref{eq:Vnidefsliding} yields sums of statistics of sliding blocks.

In an abstract version of the ``big blocks, small blocks'' approach, we approximate $V_{n,i}$ by stationary sequences of random processes $\tilde{V}_{n,i}$ that are asymptotically independent.
For example, $V_{n,i}(g)=m_n^{-1/2}\sum_{j=1}^{r_n}g(X_{n,(i-1)r_n+j})$ can be approximated by $\tilde{V}_{n,i}(g)=m_n^{-1/2}\sum_{j=1}^{r_n-l_n}$ $g(X_{n,(i-1)r_n+j})$ for a suitable sequence $l_n=\ord(r_n)$.

We now list the conditions used to establish convergence of the finite dimensional marginal distributions (fidis) of $Z_n$.
\begin{itemize}
	
	\item[\bf (A)] $(X_{n,i})_{1\leq i\leq n}$ is stationary for all $n\in\mathbb{N}$ and the sequences $s_n,r_n\in\mathbb{N}$ satisfy $s_n=\ord(r_n)$ and $r_n=\ord(n)$.

	\item[\bf(V)] For all $n\in\N$, $1\le i\le m_n=\floor{(n-s_n+1)/r_n}$, $V_{n,i}$ and $\tilde V_{n,i}$ are real valued processes indexed by $\GG$ that are measurable w.r.t.\ $(X_{n,(i-1)r_n+1},...,X_{n,ir_n+s_n-1})$, and $(V_{n,i},\tilde V_{n,i})_{1\le i\le m_n}$  is stationary.
	
	\item[\bf(M$\mathbf{\tilde{V}}$)] $m_n\beta_{n,0}^{\tilde{V}}\to 0$
	
	\item[\bf(MX$_k$)] $m_n\beta_{n,(k-1)r_n-s_n}^{X}\to 0$
	

	\item[\bf($\mathbf{\Delta}$)] $\Delta_n:=V_n-\tilde V_n$ satisfies

(i)
	$\displaystyle
	E\left[(\Delta_n(g)-E[\Delta_n(g)])^2\mathds{1}_{\left\{ |\Delta_n(g)-E[\Delta_n(g)]|\leq \sqrt{p_n}\right\}}\right]=\ord\left(p_n/m_n\right), \qquad \forall\, g\in\mathcal{G},
	$
\smallskip

	(ii)
	$\displaystyle
	P\left\{|\Delta_n(g)-E[\Delta_n(g)]|> \sqrt{p_n}\right\}=\ord\left(1/m_n\right), \qquad \forall\, g\in\mathcal{G}.
	$
	
	\item[\bf (L)]
	$\displaystyle
	E\left[(V_n(g)-E[V_n(g)])^2 \mathds{1}_{\left\{ |V_n(g)-E[V_n(g)]|>\epsilon\sqrt{p_n} \right\}}\right]=\ord\left(p_n/m_n\right), \quad  \forall\, g\in\mathcal{G}, \epsilon>0.
	$
\end{itemize}
In addition, Conditions (C) and (D0) stated in Section \ref{section:sliding} are needed.
Condition ($\Delta$) ensures that the approximation of $V_{n,i}$ by $\tilde V_{n,i}$ is sufficiently accurate. It is always fulfilled if
\begin{equation}
\label{eq:bed-c1-zweites.moment}
E\left[(\Delta_n(g))^2\right]=\ord\left(p_n/m_n\right), \qquad  \forall\, g\in\mathcal{G}.
\end{equation}
The mixing conditions (MX$_k$) and (M$\tilde V$) enable us to replace the summands by independent copies, while (C) and the Lindeberg condition (L) imply convergence of  the sum of independent copies of $V_n(g)$.

\begin{theorem}
	\label{satz: emp prozess fidi}
	 Suppose the conditions (A), (V), (M$\tilde{V}$), (MX$_k$) for some $k\in\N, k\ge 2$, ($\Delta$), (L), (D0) and (C) are satisfied. Then the fidis of the empirical process $(Z_n(g))_{g\in\mathcal{G}}$ converge weakly to the fidis of a Gaussian process with covariance function $c$.
\end{theorem}

To conclude convergence of the processes $(Z_n(g))_{g\in\mathcal{G}}$, we have to show that they are  asymptotically tight or asymptotically equicontinuous. To this end, we need  (D1)--(D3) from Section~\ref{section:sliding} and the following conditions:
\begin{itemize}
	\item[\bf (B)] $E[|V_{n}(g)|^2]<\infty $ for all $g\in\GG$, and  $	V_n(\mathcal{G}):=\sup_{g\in\mathcal{G}}|V_n(g)|<\infty$.

	\item[\bf (L1)]
	$\displaystyle
	E^*\left[ V_n(\mathcal{G}) \mathds{1}_{\left\{V_n(\mathcal{G})>\epsilon\sqrt{p_n}\right\}}\right]=\ord(\sqrt{p_n}/m_n), \quad\forall\,\epsilon>0.
	$
\end{itemize}
Condition (L1) follows from  the following condition of Lindeberg type, that also implies (L) (see Supplement):
\begin{itemize}
	\item[\bf (L2)]
	$\displaystyle
	E^*\left[ (V_n(\mathcal{G}))^2 \mathds{1}_{\left\{V_n(\mathcal{G})>\epsilon\sqrt{p_n}\right\}}\right]=\ord(p_n/m_n), \quad
\forall\,\epsilon>0.
	$
\end{itemize}

\begin{theorem}
	\label{satz: asymp tight_equicon}
  \begin{itemize}
	 \item[(i)] If the conditions (A), (V), (MX$_k$) for some $k\in\N, k\ge 2$, (B), (L1), (D0), (D1) and (D2) are satisfied, then the processes $(Z_n(g))_{g\in\mathcal{G}}$ are asymptotically tight.
     \item[(ii)] If the conditions (A), (V), (MX$_k$) for some $k\in\N, k\ge 2$, (B), (L2), (D0), (D1) and (D3) are satisfied, then the processes $(Z_n(g))_{g\in\mathcal{G}}$ are asymptotically equicontinuous.
   \end{itemize}
   \smallskip

   Hence, the processes converge to a Gaussian process with covariance function $c$ if, in addition, the assumptions of Theorem \ref{satz: emp prozess fidi} are fulfilled.
\end{theorem}

\section{Proofs}
\label{section:proofs}

\subsection{Proofs of Appendix \ref{section:abstract}}
 We first show that for the proof of convergence of the fidis it suffices to consider independent copies of $V_{n,i}$.

\begin{lemma}
	\label{lemma:emp.prozess.fidi.unab}
	
	Suppose the conditions (A), ($\Delta$), (M$\tilde{V}$) and  (MX$_k$) for some $k\in\N, k\ge 2$, are satisfied. Let
   \begin{equation}
	Z_n^*(g):=\frac{1}{\sqrt{p_n}}\sum_{i=1}^{m_n}\big(V_{n,i}^*(g)-E[V_{n,i}^*(g)]\big), \qquad  g\in\mathcal{G},
	\end{equation}
	where $V_{n,i}^*$ are independent copies of $V_{n,i}$, $1\leq i\leq m_n$.
Then the fidis of $(Z_n(g))_{g\in\mathcal{G}}$ converge weakly if and only if the fidis of $(Z_n^*(g))_{g\in\mathcal{G}}$ converge, and if so, the limits coincide.
\end{lemma}
\begin{proof}[Proof of Lemma \ref{lemma:emp.prozess.fidi.unab}]
	Let $\Delta_{n,i}:=V_{n,i}-\tilde V_{n,i}$ and   $\Delta_{n,i}^*$ be independent copies of $\Delta_{n,i}$, $1\le i\le m_n$. For the $k$ for which (MX$_k$) is satisfied and for
	all $i\in\{1,\ldots,k\}$, condition ($\Delta$) and Theorem 1 of Section IX.1 of \cite{petrov1975} yield
   \begin{equation}
	\label{eq: delta k unab kon 0}
	\frac{1}{\sqrt{p_n}} \sum_{j=1}^{m_{n,k,i}} \big(\Delta_{n,jk-i}^*(g)-E[\Delta_{n,jk-i}^*(g)]\big)=\ord_P(1), \qquad  \forall\, g\in\mathcal{G},
	\end{equation}
    where $m_{n,k,i}:=\floor{(m_n+i)/k}\le m_n$.
	
	Recall that	$\Delta_{n,jk-i}$ is measurable w.r.t.\ $(X_{n,(jk-i-1)r_n+1},\ldots,X_{n,(jk-i)r_n+s_n-1})$. For different $j$, these blocks are separated by at least $(k-1)r_n-s_n$ observations. Hence, by (MX$_k$) and Lemma 2 of \cite{eberlein1984}, the total variation distance between the joint distribution of $\Delta_{n,jk-i}$, $1\leq j\leq m_{n,k,i}$, and that of $\Delta_{n,jk-i}^*$, $1\leq j\leq m_{n,k,i}$, converges to 0:
	\begin{equation}
	\|P^{ (\Delta_{n,jk-i}^*)_{1\leq j\leq m_{n,k,i}} }- P^{ (\Delta_{n,jk-i})_{1\leq j\leq m_{n,k,i}} } \|_{TV}
	\leq m_{n,k,i} \beta_{n,(k-1)r_n-s_n}^X \rightarrow 0.   \label{eq: eberlein fuer delta}
	\end{equation}
	Hence \eqref{eq: delta k unab kon 0} holds with  $\Delta_{n,jk-i}$ instead of $\Delta_{n,jk-i}^*$.	Summing over all $i\in\{0,\ldots,k-1\}$ leads to 	
	\begin{equation}
	\label{eq: delta kon 0}
	\frac{1}{\sqrt{p_n}} \sum_{j=1}^{m_{n}} \big(\Delta_{n,j}(g)-E\Delta_{n,j}(g)\big)=\ord_P(1),\qquad \forall\, g\in\mathcal{G}.
	\end{equation}
	Thus the fidis of $\tilde{Z}_n$ defined by
    \begin{equation}
      \tilde{Z}_n(g) := \frac{1}{\sqrt{p_n}}\sum_{i=1}^{m_n}\big(\tilde{V}_{n,i}(g)-E\tilde{V}_{n,i}(g)\big)
   = Z_n(g) - \frac{1}{\sqrt{p_n}} \sum_{j=1}^{m_{n}} \left(\Delta_{n,j}(g)-E\Delta_{n,j}(g)\right), \quad g\in\GG,
   \end{equation}
   converge if and only if the fidis of $Z_n$ converge, and the limits coincide if they exist.

   Now, by assumption (M$\tilde V$) and again the inequality by \cite{eberlein1984},
   \begin{equation}
	\|P^{ (\tilde{V}_{n,j}^*)_{1\leq j\leq m_{n}} }- P^{ (\tilde{V}_{n,j})_{1\leq j\leq m_{n}} } \|_{TV} \leq m_n \beta_{n,0}^{\tilde{V}} \rightarrow 0.
	\end{equation}
    where $\tilde V_{n,i}^*$ are iid copies of $\tilde V_{n,i}$. Hence, the fidis of $\tilde V_n$ converge if and only if the fidis of
    \begin{equation}
      \tilde{Z}_n^*(g) := \frac{1}{\sqrt{p_n}}\sum_{i=1}^{m_n}\big(\tilde{V}_{n,i}^*(g)-E\tilde{V}^*_{n,i}(g)\big), \quad g\in\GG,
   \end{equation}
   converge. Finally using the analog to \eqref{eq: delta kon 0} with $\Delta_{n,j}^*$ instead of $\Delta_{n,j}$, we arrive at the assertion.
\end{proof}

\begin{proof}[Proof of Theorem \ref{satz: emp prozess fidi}]
	 In view of the assumptions (L) and (C),  the multivariate central limit theorem by Lindeberg-Feller yield convergence of the fidis of $(Z_n^*(g))_{g\in\mathcal{G}}$. The assertion thus follows from Lemma \ref{lemma:emp.prozess.fidi.unab}.
\end{proof}

\begin{proof}[Proof of Theorem \ref{satz: asymp tight_equicon}]

	It suffices to prove that, for all $i\in\{0,\ldots, k-1\}$, the processes
	\begin{equation}
	\label{eq:zn(i) definition}
	Z_n^{(i)}(g)=\frac{1}{\sqrt{p_n}}\sum_{j=1}^{m_{n,k,i}}\left(V_{n,kj-i}(g)-EV_{n,kj-i}(g)\right), \quad g\in\GG,
	\end{equation}
	(with $m_{n,k,i}:=\floor{(m_n+i)/k}$) are asymptotically tight or asymptotically equicontinuous, respectively,  since these properties carry over to their sum $Z_n$.

    By the same arguments as in the proof of Lemma \ref{lemma:emp.prozess.fidi.unab} (cf.\ \eqref{eq: eberlein fuer delta}), we may conclude
    \begin{equation}
	\label{eq:eberlein fuer V}
	\|P^{(V_{n,jk-i}^*)_{1\leq j\leq m_{n,k,i}}}-P^{(V_{n,jk-i})_{1\leq i\leq m_{n,k,i}}}\|_{TV}\leq m_{n,k,i} \beta^X_{n,(k-1)r_n-s_n}\to 0,
	\end{equation}
	where $V_{n,jk-i}^*$, ${1\leq j\leq m_{n,k,i}}$ are independent copies of $V_{n,1}$.
	
    Therefore, it suffices to prove asymptotic tightness or asymptotic equicontinuity, respectively, of
	\begin{equation}
	Z_n^{(i)*}(g)=\frac{1}{\sqrt{p_n}}\sum_{j=1}^{m_{n,k,i}}\left(V_{n,kj-i}^*(g)-E[V_{n,kj-i}^*(g)]\right), \quad g\in\GG.
	\end{equation}
	This, however, follows under the given conditions (B), (L1), (D1) and (D2) from Theorem 2.11.9, and under the given conditions (B), (L2), (D0), (D1) and (D3) from Theorem 2.11.1 of \cite{vanderVaart1996}. Note that the measurability condition of the latter theorem is automatically fulfilled if the processes are separable.
\end{proof}

\subsection{Proofs of Section \ref{section:sliding}}
\begin{proof}[Proof of Theorem \ref{cor: g besch bedingung sliding blocks}]
    First check that, by assumption (A2),
   \begin{align}
	E^*\bigg[  \sup_{g\in\GG} (Z_n(g)-\bar Z_n(g))^2 \bigg]
    & = E^* \bigg[  \sup_{g\in\GG} \bigg(\frac 1{\sqrt{p_n}b_n(g)}\sum_{j=r_nm_n+1}^{n-s_n} \big(g(Y_{n,j})-E(g(Y_{n,j})) \big) \bigg)^2 \bigg]\\
    & \le  \frac{\|g_{\max}\|_{\infty}^2 r_n^2}{p_n \inf_{g\in\mathcal{G}}b_n(g)^2} \to 0,
	\end{align}
   which implies \eqref{eq:sliding fidi con wie ursprung proz}.

   To prove convergence of the fidis, we apply Theorem \ref{satz: emp prozess fidi} to
	 $V_{n,i}$ defined by \eqref{eq:Vnidefsliding} and
	$\tilde{V}_{n,i}(g)=(b_n(g))^{-1}\sum_{j=1}^{r_n-l_n} g(Y_{n,(i-1)r_n+j})$, for which condition (V) is obvious.
   The conditions (M$\tilde V$) and (MX$_2$) follow readily from (MX) and $\ell_n=\ord(r_n)$.

   For the above choices, we obtain
	\begin{equation}
	\Delta_{n,1}(g)=V_{n,1}(g)-\tilde{V}_{n,1}(g)=\frac{1}{b_n(g)}\sum_{j=r_n-l_n+1}^{r_n} g(Y_{n,j})\overset{d}{=}\frac{1}{b_n(g)}\sum_{j=1}^{l_n} g(Y_{n,j}).
	\end{equation}
   Using the arguments of the proof of Cor.\ 3.6 of \cite{drees2010} with $X_{n,i}$ replaced by $g(Y_{n,i})$ (cf.\ also the proof of Theorem \ref{cor:g unbounded sliding block con}), we see that
   \begin{align}
	E\left[(\Delta_n(g))^2\right]
	&\leq  \frac{1}{b_n(g)^2}\|g_{\max}\|_{\infty}^2 E\bigg[\bigg(\sum_{j=1}^{l_n} \mathds{1}_{\left\{g(Y_{n,j}) \neq 0\right\}} \bigg)^2\bigg]\\
  & =\Ord\bigg(\frac{l_n}{r_nb_n(g)^2}E\bigg(\sum_{j=1}^{r_n} \mathds{1}_{\left\{g(Y_{n,j}) \neq 0\right\}} \bigg)^2\bigg)\\
  & = \ord\bigg(\frac{p_n}{m_n}\bigg).  \label{eq:Deltan2ndmombound}
	\end{align}
   Hence, Condition \eqref{eq:bed-c1-zweites.moment} is fulfilled, which in turn implies Condition ($\Delta$).

	Since $g_{\max}$ is bounded and $\inf_{g\in\mathcal{G}}b_n(g)>0$, we have
	\begin{equation} \label{eq:VnGbound}
	V_n(\mathcal{G})=\sup_{g\in\mathcal{G}}\frac{1}{b_n(g)}\sum_{j=1}^{r_n} g(Y_{n,(i-1)r_n+j}) \leq r_n\|g_{\max}\|_{\infty}\frac{1}{\inf_{g\in\mathcal{G}} b_n(g)}
	< \infty.
	\end{equation}
	Because of $r_n=\ord(\sqrt{p_n}\inf_{g\in\mathcal{G}}b_n(g))$, for all $\epsilon>0$, eventually $V_n(\mathcal{G})\leq  \sqrt{p_n}\epsilon$, so that Condition (L2) (and thus (L), too) is trivial.
	Now the assertion follows from Theorem~\ref{satz: emp prozess fidi}.
\end{proof}

\begin{proof}[Proof of Lemma \ref{lemma:pratts lemma bedingung allgemeines setting}]
   The stationarity assumption (A1) and condition (S) imply
	\begin{align}
	E\bigg(\sum_{j=1}^{r_n} \mathds{1}_{\left\{g(Y_{n,j}) \neq 0\right\}}\bigg)^2
	&= \sum_{i=1}^{r_n}\sum_{j=1}^{r_n}E\left[ \mathds{1}_{\left\{g(Y_{n,i}) \neq 0\right\}}\mathds{1}_{\left\{g(Y_{n,j}) \neq 0\right\}}  \right]\\
	&\le 2 r_n \sum_{k=1}^{r_n}\left(1-\frac{k-1}{r_n}\right) P\left\{ g(Y_{n,1}) \neq 0,  g(Y_{n,k}) \neq 0  \right\}\\
     &= \Ord\bigg(\frac{p_n b_n(g)^2}{m_n}\bigg).
	\end{align}
   \end{proof}

\begin{proof}[Proof of Theorem \ref{cor: g besch bedingung sliding blocks process}]
    Because of \eqref{eq:sliding fidi con wie ursprung proz}, the convergence of $Z_n$ and the convergence of $\bar Z_n$ are equivalent. To prove the former, we apply Theorem \ref{satz: asymp tight_equicon} to the processes $V_{n,i}$ and $\tilde V_{n,i}$ defined in the proof of Theorem \ref{cor: g besch bedingung sliding blocks}. Since the conditions (V), (MX$_2$) and (L2) have already been verified there and the (D$i$)-conditions, $i\in\{0,1,2,3\}$, are explicitly assumed to hold in Theorem \ref{cor: g besch bedingung sliding blocks process}, it remains to show that (B) holds. This, however, is obvious from \eqref{eq:VnGbound}.
\end{proof}

\begin{proof}[Proof of Theorem \ref{cor:g unbounded sliding block con}]
 We again apply Theorem \ref{satz: emp prozess fidi} to establish fidi-convergence of $(Z_n(g))_{g\in\mathcal{G}}$. Only the conditions ($\Delta$) and (L) must be verified, because the remaining conditions follow as in the proof of Theorem \ref{cor: g besch bedingung sliding blocks}.

	By the H\"{o}lder inequality, the generalized Markov inequality and \eqref{eq: bed g messbar bedingung sliding blocks}, for all $g\in\GG$, we obtain
	\begin{align}
	E\Big[ &(V_n(g))^2 \mathds{1}_{\left\{|V_n(g)|>\sqrt{p_n}\epsilon\right\}}\Big]
	\\
	&=\frac{1}{b_n^2(g)}E\bigg[ \bigg(\sum_{i=1}^{r_n}g(Y_{n,i})\bigg)^2 \mathds{1}_{\left\{\big|\sum_{i=1}^{r_n}g(Y_{n,i})\big|>\sqrt{p_n}b_n(g)\epsilon\right\}}\bigg]\\
	&\leq\frac{1}{b_n^2(g)} \bigg(E\bigg[ \bigg|\sum_{i=1}^{r_n}g(Y_{n,i})\bigg|^{2+\delta}\bigg]\bigg)^{2/(2+\delta)} \bigg(E\Big[\mathds{1}_{\big\{\big|\sum_{i=1}^{r_n}g(Y_{n,i})\big|>\sqrt{p_n}b_n(g)\epsilon \big\}}\Big]\bigg)^{\delta/(2+\delta)}\\
	&\le \frac{1}{b_n^2(g)} \bigg(E\bigg[ \bigg|\sum_{i=1}^{r_n}g(Y_{n,i})\bigg|^{2+\delta}\bigg]\bigg)^{2/(2+\delta)}
  \bigg(\frac{E\big[ |\sum_{i=1}^{r_n}g(Y_{n,i})|^{2+\delta}\big]}{(\sqrt{p_n}b_n(g)\epsilon)^{2+\delta}} \bigg)^{\delta/(2+\delta)}\\
	&= \Ord\bigg(\frac{1}{b_n^2(g)}\cdot\frac{p_nb_n^2(g)}{m_n} \cdot\frac{1}{(\sqrt{p_n}b_n(g))^{\delta}}\bigg)
	 =\ord\bigg(\frac{p_n}{m_n}\bigg),
	\end{align}
	because $\sqrt{p_n}b_n(g)\to \infty$ by assumption (A2). It is easily seen (cf.\ Section 7 in the Supplement) that this bound implies condition (L).

   Furthermore,
   \begin{align}
   E\bigg[\bigg(\sum_{i=1}^{r_n} |g(Y_{n,i})|\Bigg)^2\bigg]&
    \ge \sum_{j=1}^{\floor{r_n/l_n}} E\bigg[\bigg(\sum_{i=1}^{l_n} |g(Y_{n,(j-1)l_n+i})|\bigg)^2\bigg]
    &= \floor{r_n/l_n} E\bigg[\bigg(\sum_{i=1}^{l_n} |g(Y_{n,i})|\bigg)^2\bigg]
   \end{align}
   and thus, by \eqref{eq: bed g messbar bedingung sliding blocks},
   \begin{align}
      E(\Delta_n(g)^2) & \le \frac 1{b_n^2(g)} E\bigg[\bigg(\sum_{i=1}^{l_n} |g(Y_{n,i})|\bigg)^2\bigg]\\
        & \le \frac 1{b_n^2(g) \floor{r_n/l_n}}E\bigg[\bigg(\sum_{i=1}^{r_n} |g(Y_{n,i})|\bigg)^2 \bigg]\\
        & \le \frac 1{b_n^2(g) \floor{r_n/l_n}} E\bigg[\bigg(\sum_{i=1}^{r_n} |g(Y_{n,i})|\bigg)^{2+\delta}+\ind{\big\{\sum_{i=1}^{r_n} |g(Y_{n,i})|\ne 0\big\}}\bigg] \\
        & = \Ord\bigg( \frac{l_n}{r_n b_n^2(g)}\bigg( \frac{p_n b_n^2(g)}{m_n} +P\{V_n(|g|)\ne 0\}\bigg)\bigg)=\ord\bigg( \frac{p_n}{m_n}\bigg) 
   \end{align}
   where in the last step we have used the assumption $m_n l_nP\{V_n(|g|)\ne 0\}=\ord(r_nb_n^2(g)p_n)$ for all $g\in\mathcal{G}$. Hence, condition \eqref{eq:bed-c1-zweites.moment} holds, which in turn implies ($\Delta$). Now, the convergence of the fidis of $(Z_n(g))_{g\in\mathcal{G}}$ follows from Theorem \ref{satz: emp prozess fidi}.

  Similarly,
	\begin{align}
E \Big((\bar Z_n(g)-Z_n(g))^2\Big)
	& = Var\bigg[\frac{1}{\sqrt{p_n}b_n(g)}\sum_{j=r_nm_n+1}^{n-s_n} g(Y_{n,j})\bigg]\\
	&\leq \frac{1}{p_n b_n^2(g)} E\bigg[\bigg( \sum_{j=r_nm_n+1}^{n-s_n} |g(Y_{n,j})| \bigg)^2\bigg]\\
   & = \Ord\bigg( \frac{1}{m_n}+\frac{P\{V_n(|g|)\ne 0\}}{p_n b_n^2(g)}\bigg)\to 0,
	\end{align}
because $p_nb_n^2(g)\to\infty$ by assumption (A2), so that the fidi-convergence of $(\bar{Z}_n(g))_{g\in\mathcal{G}}$ follows, too.

Under the conditions of part (ii), the above calculations with $g_{\max}$ instead of $g$ yield \eqref{eq:sliding fidi con wie ursprung proz} as well as
$$ E^*\Big[(V_n(\GG))^2\ind{\{V_n(\GG)>\sqrt{p_n}\epsilon\}}\Big] =\frac{1}{b_n^2}E\bigg[ \bigg(\sum_{i=1}^{r_n}g_{\max}(Y_{n,i})\bigg)^2 \mathds{1}_{\left\{\sum_{i=1}^{r_n}g_{\max}(Y_{n,i})>\sqrt{p_n}b_n\epsilon\right\}}\bigg]=\ord\Big(\frac{p_n}{m_n}\Big),
$$
i.e.\ (L2). Since Condition (B) is obvious, the assertion follows from Theorem \ref{satz: asymp tight_equicon}.
\end{proof}


\subsection{Proofs of Subsection \ref{section:sliidng.vs.disjoint}}
\begin{proof}[Proof of Theorem \ref{lemma:slidng vs disjoint zaehler}]
		We compare the pre-asymptotic variances which converge to $c^{(d)}$ and $c^{(s)}$, respectively. Check that, by stationarity,
	\begin{align}
	\frac{1}{r_nv_na_n^2}& Var\bigg(\sum_{i=1}^{r_n/s_n}g(Y_{n,is_n+1})\bigg)\\
	&=\frac{1}{r_nv_na_n^2} E\bigg[\sum_{i=1}^{ r_n/s_n}\sum_{j=1}^{ r_n/s_n}g(Y_{n,js_n+1})g(Y_{n,is_n+1}) \bigg]-\frac{1}{r_nv_na_n^2}\Big(\frac{r_n}{s_n} E[g(Y_{n,0})]\Big)^2\\
	&=\frac{1}{r_nv_na_n^2} \sum_{k=- r_n/s_n+1}^{ r_n/s_n-1}\left( \frac{r_n}{s_n}-|k|\right) E\left[g(Y_{n,ks_n})g(Y_{n,0}) \right] -\frac{r_n E[g(Y_{n,0})]^2}{s_n^2v_na_n^2} \\
	&=\frac{1}{s_nv_na_n^2} \sum_{l=-r_n+1}^{r_n-1}\mathds{1}_{\{l \text{ mod }s_n =0 \}}\left(1-\frac{|l|}{r_n}\right) E\left[g(Y_{n,l})g(Y_{n,0}) \right] -\frac{r_n E[g(Y_{n,0})]^2}{s_n^2v_na_n^2}.
	\end{align}
	Similarly
	\begin{align}
	\frac{1}{r_nv_ns_n^2a_n^2} & Var\bigg(\sum_{i=1}^{r_n}g(Y_{n,i})\bigg) \\
& =\frac{1}{v_ns_n^2a_n^2}\sum_{k=-r_n+1}^{r_n-1}\left(1-\frac{|k|}{r_n}\right)E\left[g(Y_{n,0})g(Y_{n,k})\right] -\frac{r_nE\left[g(Y_{n,0})\right]^2}{v_ns_n^2a_n^2}.
	\end{align}
	In view of \eqref{eq:csdef} and \eqref{eq:cddef}, it suffices to show that the
	difference between these pre-asymptotic variances
	\begin{align}
	\frac{1}{s_nv_na_n^2}\bigg( \sum_{k=-r_n+1}^{r_n-1}\left(1-\frac{|k|}{r_n}\right)\gamma_n(k)E\left[g(Y_{n,0})g(Y_{n,k})\right] \bigg)
	\end{align}
	is non-negative. Here
	\begin{equation}
	\gamma_n(k)=\begin{cases}
	1-\frac{1}{s_n},& \text{if } k \text{ mod }s_n=0,\\
	-\frac{1}{s_n},& \text{if } k \text{ mod }s_n\neq 0,
	\end{cases}
	\end{equation}
for $k\in \mathbb{Z}$. To this end, we take up an idea by \cite{zou2019}, proof of Lemma A.10.
	
	Let $U_n$ be uniformly distributed on $\{0,\ldots,s_n-1\}$ and  independent of $(X_{n,i})_{1\le i\le n}$. Define
	\begin{align}
	\phi_{n,k}=\begin{cases}
	\frac{s_n-1}{\sqrt{s_n}},& \text{if } k \text{ mod } s_n=U_n,\\
	-\frac{1}{\sqrt{s_n}},& \text{ else,}
	\end{cases}
	\end{align}
	for $k\in\mathbb{Z}$. If $(h \text{ mod }s_n)=0$ then
	\begin{equation}
	E[\phi_{n,k}\phi_{n,k+h}]=\frac{1}{s_n} \cdot \frac{(s_n-1)^2}{s_n}+\frac{s_n-1}{s_n} \cdot \frac{1}{s_n}=1-\frac{1}{s_n},
	\end{equation}
	whereas for $(h \text{ mod }s_n)\neq 0$
	\begin{equation}
	E[\phi_{n,k}\phi_{n,k+h}]=\frac{2}{s_n}  \cdot \frac{s_n-1}{\sqrt{s_n}} \cdot \frac{-1}{\sqrt{s_n}}+\frac{s_n-2}{s_n} \cdot \frac{1}{s_n}=-\frac{1}{s_n}.
	\end{equation}
	Thus, $E[\phi_{n,k}\phi_{n,k+h}]=\gamma_n(h)$ and
	\begin{equation}
	E[\phi_{n,j}\phi_{n,i}g(Y_{n,i})g(Y_{n,j})] =E[\phi_{n,j}\phi_{n,i}]E[g(Y_{n,i})g(Y_{n,j})]
	=\gamma_{n}(|i-j|)E[g(Y_{n,0})g(Y_{n,|i-j|})]
	\end{equation}
	for all $i,j\in\{1,...,r_n\}$, since $U_n$ and $(X_1,...,X_n)$ are independent.
	Similarly as above, we conclude
	\begin{align}
	0&\leq \frac{1}{r_n}E\bigg[\bigg(\sum_{j=1}^{r_n}\phi_{n,j} g(Y_{n,j})\bigg)^2 \bigg]
	= \frac{1}{r_n}\sum_{j=1}^{r_n}\sum_{i=1}^{r_n}\gamma_{n}(|i-j|)E[g(Y_{n,0})g(Y_{n,|i-j|})]\\
	&= \sum_{k=-r_n+1}^{r_n-1}\left(1-\frac{|k|}{r_n}\right)\gamma_{n}(|k|)E[g(Y_{n,0})g(Y_{n,k})],
	\end{align}
	which proves the assertion.
\end{proof}

\begin{proof}[Proof of Theorem \ref{satz:sliding vs disjoint ganzer schaetzer}]
	W.l.o.g.\ we assume $g\ge 0$ which implies $\xi\ge 0$.
	Since $\tilde{c}^{(d)}-\tilde{c}^{(s)}=c^{(d)}-c^{(s)}-2\xi (c^{(d,v)}-c^{(s,v)})$, in view of Theorem \ref{lemma:slidng vs disjoint zaehler} it suffices to show that $c^{(d,v)}\le c^{(s,v)}$.
    Using the row-wise stationarity of the triangular scheme, the asymptotic covariance $c^{(s,v)}$ can be calculated as the limit of
   \begin{align}
   \frac{1}{r_nv_ns_na_n} & Cov\Big(\sum_{j=1}^{r_n}g(Y_{n,j}),\sum_{i=1}^{r_n}\mathds{1}_{\{X_{n,i}\neq 0\}}\Big)\\
   &= \frac{1}{r_nv_ns_na_n}\sum_{j=1}^{r_n}E\Big[g(Y_{n,j})\sum_{i=1}^{r_n}\mathds{1}_{\{X_{n,i}\neq 0\}}\Big] - \frac 1{r_nv_ns_n a_n}\cdot r_n Eg(Y_{n,1})\cdot r_nv_n\\
   &= \frac{1}{r_nv_ns_na_n}\sum_{j=1}^{r_n}E\Big[g(Y_{n,1})\sum_{i=2-j}^{r_n-j+1}\mathds{1}_{\{X_{n,i}\neq 0\}}\Big] + \frac { r_n Eg(Y_{n,1})}{s_n a_n}.  \label{eq:rhscsv}
  	\end{align}
Likewise,  $c^{(d,v)}$ is the limit of
	\begin{align}
	 \frac{1}{r_nv_na_n} & Cov\Big(\sum_{k=1}^{r_n/s_n}g(Y_{n,(k-1)s_n+1}),\sum_{i=1}^{r_n}\mathds{1}_{\{X_{n,i}\neq 0\}}\Big)\\
	&=\frac{1}{r_nv_na_n}\sum_{k=1}^{r_n/s_n}E\Big[g(Y_{n,1})\sum_{i=1-(k-1)s_n}^{r_n-(k-1)s_n}\mathds{1}_{\{X_{n,i}\neq 0\}}\Big]+ \frac { r_n Eg(Y_{n,1})}{s_n a_n}\\
    & = \frac{1}{r_nv_ns_na_n}\sum_{j=1}^{r_n}E\Big[g(Y_{n,1})\sum_{i=1-\floor{\frac{j-1}{s_n}}s_n}^{r_n-\floor{\frac{j-1}{s_n}}s_n}\mathds{1}_{\{X_{n,i}\neq 0\}}\Big] +\frac { r_n Eg(Y_{n,1})}{s_n a_n}. \label{eq:rhscdv}
	\end{align}
	
	It remains to show that the limit superior of the following difference between both right hand sides of \eqref{eq:rhscdv} and \eqref{eq:rhscsv} is not positive. To this end, note that
	\begin{align}
	 \frac{1}{r_nv_ns_na_n} & \sum_{j=1}^{r_n}E\bigg[g(Y_{n,1})\Big(\sum_{i=1-\floor{\frac{j-1}{s_n}}s_n}^{r_n-\floor{\frac{j-1}{s_n}}s_n}\mathds{1}_{\{X_{n,i}\neq 0\}}-\sum_{i=2-j}^{r_n-j+1}\mathds{1}_{\{X_{n,i}\neq 0\}}\Big)\bigg]\\
        & \le \frac{1}{r_nv_ns_na_n} \sum_{j=2}^{r_n} \sum_{i=r_n-j+2}^{r_n-\floor{\frac{j-1}{s_n}}s_n} E\big(g(Y_{n,1})\mathds{1}_{\{X_{n,i}\neq 0\}}\big)\\
        & =  \frac{1}{r_nv_n s_n a_n} \sum_{i=2}^{r_n}\sum_{j=r_n-i+2}^{(\floor{\frac{r_n-i}{s_n}}+1)s_n} E\big(g(Y_{n,1})\mathds{1}_{\{X_{n,i}\neq 0\}}\big)\\
        & \le \frac{1}{r_nv_n a_n} \sum_{i=2}^{r_n} E\big(g(Y_{n,1})\mathds{1}_{\{X_{n,i}\neq 0\}}\big). \label{eq:covdiff}
	\end{align}

	Note that $E g(Y_{n,1})= \Ord(s_n a_n v_n)$  by \eqref{eq:slidingexpectconv}. Using
    \begin{align}
       E\big(g(Y_{n,1})\mathds{1}_{\{X_{n,i}\neq 0\}}\big)
       & \le E g(Y_{n,1})P\{X_{n,i}\neq 0\}+2\|g\|_\infty\beta_{n,i-s_n-1}^X \\
       & =\Ord(s_n a_n v_n^2)+2\|g\|_\infty\beta_{n,i-s_n-1}^X
    \end{align}
    for $i> s_n+ k_n$ (see \cite{doukhan1994}, Section 1.2, Lemma 3 and Section 1.1, Prop.\ 1)
    and $E\big(g(Y_{n,1})\mathds{1}_{\{X_{n,i}\neq 0\}}\big)\le \|g\|_\infty v_n$ for $i\le s_n+k_n$,  we conclude that \eqref{eq:covdiff} is bounded by
    $$ \frac{s_n+k_n}{r_n a_n}\|g\|_\infty + \Ord(s_nv_n)+\frac{2\|g\|_\infty}{r_nv_n a_n} \sum_{l=k_n}^{r_n} \beta_{n,l}^X $$
    which tends to 0 under the given conditions.
    (In fact, similarly one can establish a lower bound on the difference between the pre-asymptotic covariances which shows that the difference tends to 0.)
\end{proof}

\subsection{Proofs of Section \ref{section:extremal.index}}

If ($\theta$1) and ($\theta$P) hold for some sequence $r_n$, then the former is obviously fulfilled by $r_n^*:=\floor{r_n/s_n}s_n$, too, and ($\theta$P) remains true because of
$$ \sum_{k=r_n^*+1}^{r_n} P(X_k>u_n|X_0>u_n) \le \frac{s_n}{v_n}(v_n^2+\beta_{n,r_n^*}^X)\le r_nv_n+\frac{n}{r_n}\beta_{n,s_n}\frac{r_n^2}{nv_n}\to 0. $$
Moreover, the arguments given in Subsection 2.1 show that the limit $c$ in ($\theta$2) does not change if we replace $r_n$ with $r_n^*$. Thus, w.l.o.g.\ we may assume that $r_n/s_n$ is a natural number (tending to $\infty$) for all $n\in\N$.

For all three estimators, we first prove joint convergence of a bivariate vector with components related to the numerator and the denominator, respectively, using the general theory developed in Section \ref{section:sliding} and Appendix \ref{section:abstract}.

We start with analyzing the disjoint blocks estimator using Theorem \ref{satz: emp prozess fidi}.
For $i\in\{1,\ldots,m_n\}$ with $m_n=\floor{(n-s_n+1)/r_n}$, let
\begin{align}
V_{n,i}^d &:=\frac{1}{\sqrt{m_n}}\sum_{j=1}^{r_n/s_n} \mathds{1}_{\{M_{(i-1)r_n+(j-1)s_n+1,(i-1)r_n+j s_n}> u_n\}},\\ \tilde{V}_{n,i}^d &:=\frac{1}{\sqrt{m_n}}\sum_{j=1}^{r_n/s_n-1} \mathds{1}_{\{M_{(i-1)r_n+(j-1)s_n+1,(i-1)r_n+j s_n}> u_n\}},\\
V_{n,i}^c & :=\frac{1}{\sqrt{m_n}}\sum_{j=1}^{r_n}\mathds{1}_{\{X_{(i-1)r_n+j}>u_n\}},\\
\tilde{V}_{n,i}^c  & := \frac{1}{\sqrt{m_n}}\sum_{j=1}^{r_n-s_n}\mathds{1}_{\{X_{(i-1)r_n+j}>u_n\}}.
\end{align}

Let $p_n=P\{M_{1,r_n}>u_n\}$. Recall that, under the conditions ($\theta$1) and ($\theta$P), \eqref{eq:maximum zu einzelueberschreitung} holds for all $k_n\to\infty$, $k_n\le r_n$, which in turn yields
\begin{equation} \label{eq:pnasymp}
p_n = r_n v_n (\theta+\ord(1)), \qquad P\{M_{1,s_n}>u_n\} = s_nv_n (\theta+\ord(1)),
\end{equation}
with $v_n=P\{X_1>u_n\}$.

\begin{proposition}
	\label{prop:extremal index disjoint zaehler und nenner}
	If the conditions ($\theta$1), ($\theta$2) and ($\theta$P) are satisfied, then
	\begin{align}
	\begin{pmatrix}
	Z^d_n\\Z^c_n
	\end{pmatrix}:=\begin{pmatrix}
	\frac{1}{\sqrt{p_n}}\sum_{i=1}^{m_n}\left(V_{n,i}^d-E[V_{n,i}^d]\right) \\
	\frac{1}{\sqrt{p_n}}\sum_{i=1}^{m_n}\left(V_{n,i}^c-E[V_{n,i}^c]\right)  \\
	\end{pmatrix} \xrightarrow{w} \begin{pmatrix}
	Z^d\\Z^c
	\end{pmatrix} \sim \mathcal{N}_2\left(0,\begin{pmatrix}	1 & 1/\theta\\
	1/\theta & c/\theta
	\end{pmatrix}\right).
	\end{align}
\end{proposition}

\begin{proof}
	The conditions  (A), (V), (M$\tilde{V}$) and (MX$_2$) follow readily from ($\theta$1). It thus suffices to verify the conditions ($\Delta$), (L) (which can be checked separately for $V_{n,i}^d$ and $V_{n,i}^c$) and (C), in order to conclude the assertion from Theorem \ref{satz: emp prozess fidi}.
	
	Check that
	\begin{align}
	\Delta_{n}^d&:=V_{n,1}^d-\tilde{V}_{n,1}^d= \frac{1}{\sqrt{m_n}}\mathds{1}_{\{M_{r_n-s_n+1,r_n}> u_n\}}
	\overset{d}{=} \frac{1}{\sqrt{m_n}}\mathds{1}_{\{M_{1,s_n}> u_n\}}.
	\end{align}
	Now \eqref{eq:maximum zu einzelueberschreitung} and $s_n=\ord(r_n)$ imply \eqref{eq:bed-c1-zweites.moment}, and thus ($\Delta$), for $V_{n,i}^d$:
	\begin{align}
	\frac{m_n}{p_n}E[(\Delta_{n}^d)^2]&\leq \frac{P\{M_{1,s_n}>u_n\}}{P\{M_{1,r_n}>u_n\}}\ =\frac{P\{M_{1,s_n}>u_n\}}{s_nP\{X_1>u_n\}}\cdot \frac{r_nP\{X_1>u_n\}}{P\{M_{1,r_n}>u_n\}} \cdot \frac{s_n}{r_n} \rightarrow 0.
	\end{align}
	Condition  (L) for $V_{n,i}^d$ follows immediately from
	$V_{n,i}^d\le m_n^{-1/2} r_n/s_n =\Ord\big(r_n/(s_n\sqrt{nv_n})$ $\sqrt{r_n v_n}\big)=\ord(\sqrt{p_n})$, because of \eqref{eq:pnasymp} and ($\theta$1).
	
	Since $V_{n,1}^c$ is a sliding blocks statistic with $X_{n,i}:=X_i/u_n$, bounded function $h(x_1,\ldots,x_s)$ $=1_{(1,\infty)}(x_1)$ and $b_n=\sqrt{m_n}$, the proof of Theorem \ref{cor: g besch bedingung sliding blocks} shows that ($\Delta$) and (L) hold if $r_n=\ord(\sqrt{p_n}b_n)=\ord(\sqrt{r_nv_nm_n})=\ord(\sqrt{nv_n})$ and condition \eqref{eq: bed g besch anzahl ungleich 0} is satisfied; both are immediate consequences of our assumptions ($\theta$1) and ($\theta$2).
	
	It remains to show convergence (C) of the covariance matrix. To this end, first note that by stationarity one has uniformly for all $1\le \ell\le r_n-s_n$
	\begin{align}
	\sum_{j=\ell+s_n+1}^{r_n}  P\big\{M_{\ell+1,\ell+s_n}>u_n, X_j>u_n\big\}
	& \le \sum_{j=s_n+1}^{r_n} P\{M_{1,s_n}>u_n,X_j>u_n\}\\
	& \le \sum_{i=1}^{s_n}\sum_{j=s_n+1}^{r_n} P\{X_i>u_n,X_j>u_n\}\\
	& = s_n v_n \sum_{k=1}^{r_n} \min\Big( 1,\frac{k}{s_n},\frac{r_n-k}{s_n}\Big) P(X_k>u_n|X_0>u_n)\\
	& = \ord(s_nv_n).  \label{eq:summaxbound1}
	\end{align}
	In the last step we have used Pratt's lemma \citep{pratt1960} according to which, under condition ($\theta$P), the limit of the last sum can be calculated as the infinite sum of the limit of each summand, which all equal  0, because $k/s_n\to 0$.
	Likewise,
	\begin{align}
	\sum_{j=1}^{\ell} P\big\{M_{\ell+1,\ell+s_n}>u_n, X_j>u_n\big\}
	\le \sum_{i=1}^{s_n}\sum_{j=s_n-r_n+1}^{0} P\{X_i>u_n,X_j>u_n\}
	= \ord(s_nv_n)  \label{eq:summaxbound2}
	\end{align}
	uniformly for $1\le \ell\le r_n-s_n$.
	
	By stationarity and \eqref{eq:pnasymp},
	\begin{align}
	\frac{m_n}{p_n} Var(V_n^d)
	& = \frac{r_n}{s_n p_n} P\{M_{1,s_n}>u_n\}(1-P\{M_{1,s_n}>u_n\})\\
	& \qquad +\frac 2{p_n} \sum_{1\le i<j\le r_n/s_n} Cov\big( \ind{\{M_{(i-1)s_n+1,is_n}>u_n\}},\ind{\{M_{(j-1)s_n+1,js_n}>u_n\}}\big)\\
	& = (1+\ord(1)) + \frac 2{p_n} \sum_{1\le i<j\le r_n/s_n} P\big\{M_{(i-1)s_n+1,is_n}>u_n, M_{(j-1)s_n+1,js_n}>u_n\big\}\\
	& \qquad + \Ord\bigg(\frac 1{p_n} \Big(\frac{r_n}{s_n}\Big)^2 (s_nv_n)^2\bigg).
	\end{align}
	In view of \eqref{eq:summaxbound1}, the second term can be bounded by
	$$ \frac 2{p_n} \sum_{i=1}^{r_n/s_n-1} \sum_{k=is_n+1}^{r_n} P\big\{M_{(i-1)s_n+1,is_n}>u_n,X_k>u_n\big\} =\ord\Big(\frac{r_nv_n}{p_n}\Big)=\ord(1).
	$$
	Since $(r_n/s_n)^2 (s_nv_n)^2/p_n = \Ord(r_nv_n)\to 0$ by \eqref{eq:pnasymp} and ($\theta$1), we conclude
	$$ \frac{m_n}{p_n} Var(V_n^d) \to 1. $$
	
	Next check that,  by \eqref{eq:pnasymp} and ($\theta$2),
	\begin{align}	\frac{m_n}{p_n}Var(V_{n,1}^c)&=\frac 1{p_n}Var\bigg(\sum_{j=1}^{r_n}\mathds{1}_{\{X_{j}>u_n\}}\bigg)\\
	&=\frac{r_nv_n}{p_n}\cdot\frac{1}{r_nv_n}E\bigg[\bigg(\sum_{j=1}^{r_n} \mathds{1}_{\{X_{j}>u_n\}}\bigg)^2  \bigg]- \frac 1{p_n}(r_nv_n)^2\\
	& = (1/\theta+\ord(1))(c+\ord(1))+ \Ord(r_nv_n)\\
	&\rightarrow c/\theta.  \label{eq:Varhconv}
	\end{align}
	Finally, again by \eqref{eq:pnasymp}, \eqref{eq:summaxbound1} and \eqref{eq:summaxbound2},
	\begin{align}
	\frac{m_n}{p_n} & Cov\left(V_{n,1}^d,V_{n,1}^c\right)\\
	& =\frac 1{p_n} \bigg(\sum_{i=1}^{r_n/s_n}\sum_{j=1}^{r_n} P\{M_{(i-1)s_n+1,is_n}>u_n,X_j>u_n\}-\frac{r_n}{s_n}P\{M_{1,s_n}>u_n\} r_n v_n\bigg)\\
	&   =\frac 1{p_n} \sum_{i=1}^{r_n/s_n}\bigg(s_nv_n + \sum_{j=1}^{(i-1)s_n} P\{M_{(i-1)s_n+1,is_n}>u_n,X_j>u_n\}\\
	& \hspace{3cm} +
	\sum_{j=is_n+1}^{r_n} P\{M_{(i-1)s_n+1,is_n}>u_n,X_j>u_n\}\bigg) + \Ord(r_nv_n)\\
	&   =\frac 1{p_n} \sum_{i=1}^{r_n/s_n}\big(s_nv_n +\ord(s_nv_n)\big) + \Ord(r_nv_n)\\
	& \to 1/\theta.
	\end{align}
\end{proof}

Next, we turn to the sliding blocks estimator.  Numerator and denominator can be written in terms of the process $\bar Z_n$ in \eqref{eq:defbarZn} based on $X_{n,i}:=X_i/u_n$ and the following bounded functions:
\begin{align}
g(x_1,\ldots,x_s) :=\mathds{1}_{\{\max_{1\le i\le s} x_i>1\}},\qquad
h(x_1,\ldots,x_s) :=\mathds{1}_{\{x_1>1\}}.
\end{align}
As normalizing sequences we choose $b_n(g)= \sqrt{nv_n/p_n^s}s_n $ and $b_n(h)= \sqrt{nv_n/p_n^s}$ with $p_n^s := P\{M_{1,r_n+s_n-1}>u_n\}=r_nv_n\theta(1+\ord(1))$, by \eqref{eq:maximum zu einzelueberschreitung}.
\begin{proposition}
	\label{prop:extremal index sliding zaehler und nenner}
	If the conditions ($\theta$1), ($\theta$2) and ($\theta$P) are satisfied, then
	\begin{align}
	\begin{pmatrix}
	\bar Z_n(g)\\ \bar Z_n(h)
	\end{pmatrix}&=\begin{pmatrix}
	\frac{1}{\sqrt{nv_n}s_n}\sum_{i=1}^{n-s_n+1}(\mathds{1}_{\{M_{i,i+s_n-1}> u_n\}}-p_n) \\
	\frac{1}{\sqrt{nv_n}}\sum_{i=1}^{n-s_n+1}(\mathds{1}_{\{X_i>u_n\}}-v_n) \\
	\end{pmatrix} \\
	&\xrightarrow{w} \begin{pmatrix}
	Z(g)\\Z(h)
	\end{pmatrix} \sim \mathcal{N}_2\left(0,\begin{pmatrix}	\theta  & 1 \\
	1 & c
	\end{pmatrix}\right).
	\end{align}
\end{proposition}

\begin{proof}
	We are going to apply Theorem \ref{cor: g besch bedingung sliding blocks}.
	Condition (A1) is obvious, and  (A2) with $l_n=2s_n-1$ and (MX) easily follow from ($\theta$1).
	Condition \eqref{eq: bed g besch anzahl ungleich 0} for the functional $h$ is immediate from ($\theta$2) (see proof of Proposition \ref{prop:extremal index disjoint zaehler und nenner}). To check it for $g$, we employ Lemma \ref{lemma:pratts lemma bedingung allgemeines setting}. First note that  $p_n^s b_n(g)^2/n=s_n^2v_n$. Moreover, by stationarity of the time series,
	\begin{align}
	\frac 1{s_n^2v_n} & \sum_{k=1}^{r_n} P\{M_{1,s_n}>u_n,M_{k,k+s_n-1}>u_n\}\\
	& \le \frac 1{s_n^2v_n} \sum_{k=1}^{r_n} \sum_{i=1}^{s_n} \sum_{j=k}^{k+s_n-1} P\{X_i>u_n,X_j>u_n\}\\
	& \le \frac 1{s_nv_n}  \sum_{i=1}^{s_n} \bigg(\sum_{j=1}^{s_n} P\{X_i>u_n,X_j>u_n\}+ \sum_{j=s_n+1}^{r_n+s_n-1} P\{X_i>u_n,X_j>u_n\}\bigg)\\
	& \le 1 + 2\sum_{k=1}^{s_n-1} P(X_k>u_n|X_0>u_n) + \sum_{k=1}^{r_n+s_n-2} P(X_k>u_n|X_0>u_n).
	\end{align}
	Therefore, condition (S) follows from ($\theta$P) and
	$$ \sum_{k=r_n+1}^{r_n+s_n-2} P(X_k>u_n|X_0>u_n)\le \frac{s_n}{v_n}\big(v_n^2+\beta_{n,r_n}^X \big) =\ord\Big(s_nv_n+\frac{n}{r_n}\beta_{n,r_n}^X\Big)\to 0. $$
	
	Then, condition \eqref{eq: bed g besch anzahl ungleich 0} for $g$ follows from Lemma \ref{lemma:pratts lemma bedingung allgemeines setting}. It remains to prove convergence (C) of the standardized covariance matrix. For the variance pertaining to $g$ and the covariance, this is done in Lemma \ref{lemma:covariance.extremal} (iii) and (iv). The convergence
	$$\frac{m_n}{p_n^s} Var(V_n(h)) = \frac{1+\ord(1)}{r_nv_n} Var\bigg(\sum_{j=1}^{r_n}\ind{\{X_j>u_n\}}\bigg) \to c $$
	has been shown in \eqref{eq:Varhconv}.
\end{proof}

Finally, we examine the statistics pertaining to the runs estimator, again using Theorem \ref{cor: g besch bedingung sliding blocks}. Here we consider $X_{n,i}$, and the functions $h$ defined above and
$$ f(x_1,\ldots,x_s)= \ind{\{x_1>1,\max_{2\le i\le s} x_i\le 1\}}. $$
The normalization is chosen as $b_n:=\sqrt{nv_n/p_n}$ for both functions $f$ and $h$.

\begin{proposition}
	\label{prop:extremal index runs zaehler und nenner}
	If the conditions ($\theta$1), ($\theta$2) and ($\theta$P) are satisfied, then
	\begin{align}
	\begin{pmatrix}
	\bar Z_n(f)\\ \bar Z_n(h)
	\end{pmatrix}&=\begin{pmatrix}
	\frac{1}{\sqrt{nv_n}}\sum_{i=1}^{n-s_n+1}\big(\mathds{1}_{\{X_i>u_n,M_{i+1,i+s_n-1}\leq u_n\}}-P\{X_1>u_n,M_{2,s_n}\leq u_n\}\big) \\
	\frac{1}{\sqrt{nv_n}}\sum_{i=1}^{n-s_n+1}\big(\mathds{1}_{\{X_i>u_n\}}-v_n\big)  \\
	\end{pmatrix} \\
	&\xrightarrow{w}
	\begin{pmatrix}
	\bar Z(f)\\ \bar Z(h)
	\end{pmatrix} \sim \mathcal{N}_2\left(0,\begin{pmatrix}	\theta  & 1 \\
	1  & c
	\end{pmatrix}\right).
	\end{align}
\end{proposition}

\begin{proof}[Proof of Proposition \ref{prop:extremal index runs zaehler und nenner}]
	Conditions (A1), (A2), (MX) and \eqref{eq: bed g besch anzahl ungleich 0} for the functional $h$  have already been checked in the proof of Proposition \ref{prop:extremal index sliding zaehler und nenner}. Condition \eqref{eq: bed g besch anzahl ungleich 0} for $f$ follows readily, because $f(x)\ne 0$ implies $h(x) \ne 0$. While condition (C) for $Var(V_n(h))$ has been verified in the proof of Proposition \ref{prop:extremal index sliding zaehler und nenner}, it is established for $Var(V_n(f))$ and $Cov(V_n(f),V_n(h))$ in Lemma \ref{lemma:covariance.extremal} (i) and (ii). Thus the assertion follows from 	Theorem \ref{cor: g besch bedingung sliding blocks}.
\end{proof}

Now Theorem \ref{thm:extremal index} easily follows from the above propositions by a continuous mapping argument.
\begin{proof}[Proof of Theorem \ref{thm:extremal index}]
	Since the arguments are basically the same for all three estimators, we give the details only for the disjoint blocks estimator.
	In view of $E[V_n^c] = m_n^{-1/2}r_nv_n$, $E[V_n^d]=m_n^{-1/2}(r_n/s_n) P\{M_{1,s_n}>u_n\}$ and
	$p_n^{1/2}m_n^{-1/2} (r_n v_n)^{-1}= (\theta/(nv_n))^{1/2}(1+\ord(1))=\ord(1)$ (by \eqref{eq:pnasymp} and ($\theta$1)), direct calculations show that
	\begin{align}
	\sqrt{nv_n}(\hat{\theta}_n^d-\theta)
	& =\sqrt{nv_n}\bigg(\frac{\sum_{i=1}^{m_n}V_{n,i}^d}{\sum_{i=1}^{m_n}V_{n,i}^c}-\theta\bigg)\\
	& =\sqrt{nv_n}\cdot \frac{\sqrt{p_n}(Z_n^d-\theta Z_n^c)+m_n(E[V_n^d]-\theta E[V_n^c])}{m_nE[V_n^c]+\sqrt{p_n} Z_n^c}\\
	&=\sqrt{\frac{nv_np_n}{m_n(r_nv_n)^2}}\cdot \frac{Z_n^d-\theta Z_n^c + \sqrt{m_n/p_n} r_nv_n \big( P\{M_{1,s_n}>u_n\}/(s_nv_n)-\theta\big)}{1+\sqrt{p_n/m_n} (r_nv_n)^{-1} Z_n^c}\\
	& = \sqrt{\theta}(1+\ord(1))\frac{Z_n^d-\theta Z_n^c + \Ord(\sqrt{nv_n})\big( P\{M_{1,s_n}>u_n\}/(s_nv_n)-\theta\big)}{1+\ord_P(1)}\\
	& \to \sqrt{\theta}(Z^d-\theta Z^c),
	\end{align}
	where in the last step we have used Proposition \ref{prop:extremal index disjoint zaehler und nenner} and the bias condition (B$_b$). The limit random variable is centered and normally distributed with variance $\theta(1-2\theta(1/\theta)+\theta^2(c/\theta))=\theta(\theta c-1)$.
\end{proof}

\begin{lemma}
	\label{lemma:covariance.extremal}
	If the conditions ($\theta$1), ($\theta$2) and ($\theta$P) are met, then
	\begin{enumerate}
		\item[(i)]
		\qquad $\displaystyle \lim_{n\to\infty}\frac{1}{r_nv_n}Var\bigg(\sum_{i=1}^{r_n}\mathds{1}_{\{X_i>u_n,M_{i+1,i+s_n-1}\leq u_n\}}  \bigg) = \theta$
		
		\item[(ii)]
\qquad $\displaystyle
		\lim_{n\to\infty}\frac{1}{r_nv_n}Cov\bigg(\sum_{i=1}^{r_n}\mathds{1}_{\{X_i>u_n\}}, \sum_{j=1}^{r_n}\mathds{1}_{\{X_j>u_n,M_{j+1,j+s_n-1}\leq u_n\}} \bigg) = 1
$

		\item[(iii)]
		\qquad $\displaystyle \lim_{n\to\infty}\frac{1}{r_n s_n v_n}Cov\bigg(\sum_{i=1}^{r_n}\mathds{1}_{\{M_{i,i+s_n-1}> u_n\}},\sum_{j=1}^{r_n}\mathds{1}_{\{X_j>u_n\}}\bigg) = 1
$
		
		\item[(iv)]
  \qquad $\displaystyle
		\lim_{n\to\infty}\frac{1}{r_ns_n^2v_n}Var\bigg(\sum_{i=1}^{r_n}\mathds{1}_{\{M_{i,i+s_n-1}> u_n\}}\bigg)=\theta
$
	\end{enumerate}
\end{lemma}

\begin{proof}
	
	To prove assertion (i), check that by stationarity
	\begin{align}
	\frac{1}{r_nv_n} & Var\bigg(\sum_{i=1}^{r_n}\mathds{1}_{\{X_i>u_n,M_{i+1,i+s_n-1}\leq u_n\}}  \bigg)\\
	& = \frac{1}{r_nv_n} \sum_{i=1}^{r_n} \sum_{j=1}^{r_n} P\{X_i>u_n,M_{i+1,i+s_n-1}\leq u_n,X_j>u_n,M_{j+1,j+s_n-1}\leq u_n\}\\
	& \hspace{0.5cm} - r_nv_n \big(P(M_{2,s_n}\le u_n|X_1>u_n)\big)^2\\
	& = P(M_{2,s_n}\le u_n|X_1>u_n)+\Ord(r_nv_n)\\
	& \hspace{0.5cm} + \frac{2}{r_nv_n} \sum_{i=1}^{r_n-s_n} \sum_{j=i+s_n}^{r_n} P\{X_i>u_n,M_{i+1,i+s_n-1}\leq u_n,X_j>u_n,M_{j+1,j+s_n-1}\leq u_n\},
	\end{align}
	where in the last step we have used that the probability in the sum equals 0 if $1\le |i-j|<s_n$.
	The last term is bounded by
	\begin{equation}  \label{eq:III_to_0}
	\frac{2}{r_nv_n} \sum_{i=1}^{r_n-s_n} \sum_{j=i+s_n}^{r_n} P\{X_i>u_n,X_j>u_n\} \le 2 \sum_{k=s_n-1}^{r_n} P(X_k>u_n|X_0>u_n)
	\end{equation}
	and hence it tends  to 0 by Pratt's lemma and ($\theta$P). Now \eqref{eq:runs char extremal} and $r_nv_n\to 0$ yields the convergence of the normalized variance to $\theta$.
	
	Next we consider (ii). Similarly as above, stationarity implies
	\begin{align}
	\frac{1}{r_nv_n} & Cov\bigg(\sum_{i=1}^{r_n}\mathds{1}_{\{X_i>u_n\}},\sum_{j=1}^{r_n}\mathds{1}_{\{X_j>u_n,M_{j+1,j+s_n-1}\leq u_n\}}  \bigg)\\
	& = P(M_{2,s_n}\le u_n|X_1>u_n)\\
	& \hspace{1cm} + \frac{1}{r_nv_n} \sum_{i=1}^{r_n-1} \sum_{j=i+1}^{r_n} P\{X_i>u_n,X_j>u_n,M_{j+1,j+s_n-1}\leq u_n\}\\
	& \hspace{1cm} + \frac{1}{r_nv_n} \sum_{i=s_n+1}^{r_n} \sum_{j=1}^{i-s_n} P\{X_i>u_n,X_j>u_n,M_{j+1,j+s_n-1}\leq u_n\}+\Ord(r_nv_n)\\
	& =: I+II+III+\Ord(r_nv_n),
	\end{align}
	where $I\to\theta$ by \eqref{eq:runs char extremal}. Term $III$ can be bounded by
	$ (r_nv_n)^{-1}\sum_{j=1}^{r_n-s_n} \sum_{i=j+s_n}^{r_n} P\{X_i>u_n,X_j>u_n\}$,
	which tends to 0 by \eqref{eq:III_to_0}. Moreover,
	\begin{align}
	II & = \frac{1}{r_nv_n} \sum_{i=1}^{r_n-1} \sum_{j=i+1}^{r_n} \Big( P\{X_i>u_n,X_j>u_n,M_{j+1,r_n+s_n-1}\leq u_n\}\\
	& \hspace{3cm} + P\{X_i>u_n,X_j>u_n,M_{j+1,j+s_n-1}\leq u_n,M_{j+s_n,r_n+s_n-1}> u_n \}\Big).
	\end{align}
	If first $j$ is interpreted as the last instance of an exceedance in $\{i+1,\ldots,r_n+s_n-1\}$ and then $i$ as the last instance of an exceedance in $\{1,\ldots, r_n-1\}$, then one obtains
	\begin{align}
	\frac{1}{r_nv_n} & \sum_{i=1}^{r_n-1} \sum_{j=i+1}^{r_n} P\{X_i>u_n,X_j>u_n,M_{j+1,r_n+s_n-1}\leq u_n\}\\
	& = \frac{1}{r_nv_n} \sum_{i=1}^{r_n-1} P\{X_i>u_n,M_{i+1,r_n+s_n-1}>u_n\}\\
	& = \frac{(r_n-1)v_n}{r_n v_n} - \frac{1}{r_nv_n} \sum_{i=1}^{r_n-1} P\{X_i>u_n,M_{i+1,r_n+s_n-1}\le u_n\}\\
	& = 1+\ord(1)-\frac{1}{r_nv_n}P\{M_{1,r_n-1}>u_n, M_{r_n,r_n+s_n-1}\le u_n\}\\
	& \to 1-\theta,
	\end{align}
	because of \eqref{eq:pnasymp} and $P\{M_{r_n,r_n+s_n-1}>u_n\}\le s_n v_n =\ord(r_nv_n)$. Furthermore,
	\begin{align}
	\frac{1}{r_nv_n}& \sum_{i=1}^{r_n-1} \sum_{j=i+1}^{r_n} P\{X_i>u_n,X_j>u_n,M_{j+1,j+s_n-1}\leq u_n,M_{j+s_n,r_n+s_n-1}> u_n \}\\
	& \le \frac{1}{r_nv_n} \sum_{i=1}^{r_n-1} \sum_{j=i+1}^{r_n}\big( P\{X_i>u_n,X_j>u_n\}P\{M_{j+s_n,r_n+s_n-1}> u_n \}+\beta_{n,s_n-1}^X\big)\\
	& \le r_nv_n \sum_{k=1}^{r_n} P(X_k>u_n|X_0>u_n) + \frac{r_n}{v_n}\beta_{n,s_n-1}^X\\
	& \to 0,
	\end{align}
	by ($\theta$1) and ($\theta$P). To sum up, $II\to 1-\theta$, which concludes the proof of (ii).
	
	In view of \eqref{eq:summaxbound1} and \eqref{eq:summaxbound2}, the standardized covariance in (iii) equals
	\begin{align}
	\frac{1}{r_n s_n v_n} & Cov\bigg(\sum_{i=1}^{r_n}\mathds{1}_{\{M_{i,i+s_n-1}> u_n\}},\sum_{j=1}^{r_n}\mathds{1}_{\{X_j>u_n\}}\bigg)\\
	& = \frac{1}{r_n s_n v_n} \sum_{i=1}^{r_n} \sum_{j=1}^{r_n} P\{M_{i,i+s_n-1}> u_n, X_j>u_n\} - \frac{r_n}{s_n}P\{M_{1,s_n}>u_n\} \\
	& = \frac{1}{r_n s_n v_n} \bigg(\sum_{i=1}^{r_n} \sum_{j=i}^{\min(i+s_n-1, r_n)} P\{X_j>u_n\} +\ord(r_ns_nv_n)\bigg)+ \Ord(r_nv_n)\\
	& = \frac{1}{r_n s_n} \Big( (r_n-s_n+1)s_n+\frac{s_n(s_n-1)}2\Big)+\ord(1)\\
	& \to 1.
	\end{align}
	Finally, we turn to (iv). Stationarity implies
	\begin{align}
	Var & \bigg(\sum_{i=1}^{r_n}  \mathds{1}_{\{M_{i,i+s_n-1}> u_n\}}\bigg)\\
	& = \sum_{i=1}^{r_n} \sum_{j=1}^{r_n} P\{M_{i,i+s_n-1}> u_n,M_{j,j+s_n-1}> u_n\} -\big(r_nP\{M_{1,s_n}>u_n\}\big)^2\\
	& = 2\sum_{i=1}^{r_n} \sum_{j=i}^{r_n} P\{M_{i,i+s_n-1}> u_n,M_{j,j+s_n-1}> u_n\}  -r_nP\{M_{1,s_n}>u_n\} + \Ord((r_ns_nv_n)^2)\\
	& =  2\bigg[\sum_{i=1}^{r_n-3s_n} \sum_{j=i}^{i+s_n-1} P\{M_{i,i+s_n-1}> u_n,M_{j,j+s_n-1}> u_n\}\\
	& \hspace{1.5cm}+ \sum_{i=r_n-3s_n+1}^{r_n} \sum_{j=i}^{r_n} P\{M_{i,i+s_n-1}> u_n,M_{j,j+s_n-1}> u_n\}\\
	& \hspace{1.5cm}+\sum_{i=1}^{r_n-3s_n} \sum_{j=i+s_n}^{r_n-s_n} P\{M_{i,i+s_n-1}> u_n,M_{j,j+s_n-1}> u_n\}\\
	& \hspace{1.5cm}+\sum_{i=1}^{r_n-3s_n} \sum_{j=r_n-s_n+1}^{r_n} P\{M_{i,i+s_n-1}> u_n,M_{j,j+s_n-1}> u_n\}\bigg]+\ord(r_ns_n^2v_n)\\
	& =: 2[I+II+III+IV]+\ord(r_ns_n^2v_n).
	\end{align}
	Term $II$ is of the order $s_n^2 s_nv_n=\ord(r_ns_n^2v_n)$. Term $III$ can be bounded by
	\begin{align}
	\sum_{i=1}^{r_n-3s_n}& \sum_{j=i+s_n}^{r_n-s_n} \sum_{k=j}^{j+s_n-1} P\{M_{i,i+s_n-1}> u_n,X_k>u_n\}\\
	& \le s_n \sum_{i=1}^{r_n-3s_n}\sum_{k=i+s_n}^{r_n} P\{M_{i,i+s_n-1}> u_n,X_k>u_n\}\\
	& = \ord(r_ns_n^2v_n)
	\end{align}
	by \eqref{eq:summaxbound1}. Moreover, by ($\theta$1),
	\begin{align}
	IV & \le \sum_{i=1}^{r_n-3s_n} \sum_{j=r_n-s_n+1}^{r_n}\big( P\{M_{i,i+s_n-1}> u_n\}\cdot P\{M_{j,j+s_n-1}> u_n\}+\beta_{n,s_n-1}^X\big)\\
	& = \Ord\big(r_ns_n((s_nv_n)^2+\beta_{n,s_n-1}^X)\big) = \ord(r_ns_n^2v_n)
	\end{align}
	because $r_ns_n=r_n^2s_n/r_n=\ord(nv_n s_n/r_n)=\ord(n/r_n)$.
	
	It remains to be shown that
	$$ \frac{I}{r_ns_n^2v_n}=\frac{1+\ord(1)}{s_n^2v_n}\sum_{k=1}^{s_n} P\{M_{1,s_n}> u_n,M_{k,k+s_n-1}> u_n\} \to \frac{\theta}2. $$
	Distinguish according to the last exceedance in $\{1,\ldots, s_n\}$ to conclude
	\begin{align}
	\sum_{k=1}^{s_n} & P\{M_{1,s_n}> u_n,M_{k,k+s_n-1}> u_n\}\\
	& = \sum_{k=1}^{s_n} \sum_{j=1}^{s_n} P\{X_j>u_n,M_{j+1,s_n}\le u_n,M_{k,k+s_n-1}>u_n\}\\
	& = \sum_{k=1}^{s_n} \sum_{j=k}^{s_n} P\{X_j>u_n,M_{j+1,s_n}\le u_n\}
	+\Ord\bigg(\sum_{k=1}^{s_n} \sum_{j=1}^{k-1} P\{X_j>u_n,M_{k,k+s_n-1}>u_n\}\bigg)\\
	& = \sum_{j=1}^{s_n} j P\{X_j>u_n,M_{j+1,s_n}\le u_n\} + \ord(s_n^2v_n)\\
	& = \sum_{j=1}^{s_n} j P\{X_1>u_n,M_{2,s_n-j+1}\le u_n\} + \ord(s_n^2v_n),
	\end{align}
	where in the penultimate step we  have employed \eqref{eq:summaxbound2}. The last sum can be bounded from below by
	$$ \sum_{j=1}^{s_n} j P\{X_1>u_n,M_{2,s_n}\le u_n\}= \frac{s_n(s_n+1)}2 v_n P(M_{2,s_n}\le u_n|X_1>u_n) = \frac{s_n^2 v_n}2 \theta (1+\ord(1))
	$$
	because of \eqref{eq:runs char extremal}.
	Similarly, for any sequence $t_n=\ord(s_n)$ tending to $\infty$, \eqref{eq:runs char extremal} yields the asymptotic behavior of the following upper bound
	$$ \sum_{j=1}^{s_n-t_n} j P\{X_1>u_n,M_{2,t_n}\le u_n\} +t_ns_nv_n = \frac{s_n^2 v_n}2 \theta (1+\ord(1)).
	$$
	Hence, the sum divided by $s_n^2v_n$ must tend to $\theta/2$, which concludes the proof.
\end{proof}

\medskip
{\bf Acknowledgement:} We thank two anonymous referees for their valuable comments and suggestions which lead to a substantial improvement of the presentation.

\addcontentsline{toc}{section}{References}
\bibliography{Dissbib}
\bibliographystyle{agsm}
\end{document}


\begin{frontmatter}
\title{Supplement to the paper\\Asymptotics for sliding blocks estimators of rare events}
\runtitle{Sliding blocks estimators}
\begin{aug}
  \author[A]{\fnms{Holger} \snm{Drees}\ead[label=e1]{drees@math.uni-hamburg.de}},
  \author[A]{\fnms{Sebastian} \snm{Neblung}\ead[label=e2,mark]{sebastian.neblung@uni-hamburg.de}}
  \address[A]{University of Hamburg, Department of Mathematics,
 SPST, Bundesstr.\ 55, 20146 Hamburg, Germany, \printead{e1,e2}}
\end{aug}
\begin{abstract}
We specify sufficient conditions for the asymptotic normality of the statistics considered in Subsection 2.1 of the main paper. Furthermore, a multivariate generalization to Theorem 2.5 is given and the asymptotic behavior of a sliding blocks estimator for the extremal index is analyzed which uses all observations above a {\em random} threshold. Finally we discuss the relationship between a Lindeberg type condition and an analog based on first moments.
\end{abstract}
\begin{keyword}[class=AMS]
\kwd[Primary ]{62G32}  \kwd[; secondary ]{62M10, 62G05, 60F17}
\end{keyword}

\begin{keyword}
\kwd{asymptotic efficiency, empirical processes, extremal index, extreme value analysis, sliding vs disjoint blocks,  time series, uniform central limit theorems}
\end{keyword}
\end{frontmatter}

To avoid confusion, we continue the section numbering from the main article.

\setcounter{section}{3}

\section{Conditions for asymptotic normality}

In (2.10) and (2.11), the asymptotic normality of statistics based on sliding blocks and on disjoint blocks, respectively, has been claimed. Here we give precise conditions under which these convergences follow from the general results of Section 2 and Appendix A.

For the sliding blocks statistic $T_{n}^s(g):=(n v_n s_n a_n)^{-1}\sum_{i=1}^{n-s_n+1}g(Y_{n,i})$ the asymptotic normality
\begin{equation}
\label{eq:vergle sliding stat con}
\sqrt{nv_n}(T_n^{s}(g)-E[T_n^{s}(g)])\rightarrow \mathcal{N}(0,c^{(s)}).
\end{equation}
can be established using Theorem 2.1. In the setting of Section 2, we have $m_n=\lfloor (n-s_n+1)/r_n \rfloor$, $p_n=P\{\sum_{i=1}^{r_n}g(Y_{n,i})\neq 0\}$,  and we choose $b_n(g)=b_n=\sqrt{nv_n/p_n}a_ns_n$ with $v_n=P(X_{n,1}\neq 0)$.
\begin{corollary}  \label{cor:sliding_cond}
  Suppose the following conditions are fulfilled:
\begin{itemize}
\item[(i)] $(X_{n,i})_{1\leq i\leq n}$ is stationary for all $n\in\mathbb{N}$.

\item[(ii)] The sequences $l_n,r_n,s_n\in\N$, $a_n$ and $p_n$ satisfy $s_n\le l_n=\ord(r_n)$, $r_n=\ord(n)$, $p_n\to 0$, $r_n=\ord\big(\sqrt{n v_n}s_n a_n\big)$ and $(n/r_n)\beta_{n,l_n-s_n}^{X}\to 0$.

  \item[(iii)] $g$ is measurable and bounded.
\item[(iv)]
$ \displaystyle
E\bigg[\bigg(\sum_{j=1}^{r_n} \mathds{1}_{\left\{g(Y_{n,j}) \neq 0\right\}}\bigg)^2\bigg]  =\hbox{O}\big(r_n v_n a_n^2s_n^2\big).$

\item [(v)] $c^{(s)}$ as defined in (2.9) exists in $[0,\infty)$.

\end{itemize}	
 Then convergence \eqref{eq:vergle sliding stat con} holds.
\end{corollary}
One may drop the assumption that $g$ is bounded if condition (iv) is adapted in the same way as in Theorem 2.4 (i).

Next we turn to the asymptotic normality of the statistic $T_n^{d}(g)=(n v_n a_n)^{-1}\sum_{i=1}^{\lfloor n/s_n\rfloor}$ $g(Y_{n,(i-1)s_n+1})$
based on disjoint blocks:
\begin{align}
\label{eq:vergl disjoint stat con}
&\sqrt{nv_n}\left(T_n^{d}(g)-E[T_n^{d}(g)]\right)\rightarrow \mathcal{N}(0,c^{(d)}).
\end{align}
Recall that $r_n/s_n$ is assumed to be a natural number. In the setting of Appendix A, we let $V_{n,i}(g):=\sqrt{p_n/(nv_n a_n^2)}\sum_{j=1}^{r_n/s_n}g(Y_{n,(j-1)s_n+(i-1)r_n+1})$, $1\le i\le m_n$, with $p_n=P\{\sum_{j=1}^{r_n/s_n}$ $g(Y_{n,(j-1)s_n+1})\neq 0\}$. Moreover, we choose a sequence $l_n$, $n\in\N$, of multiples of $s_n$ and define the shortened sum
$\tilde{V}_{n,i}(g)=\sqrt{p_n/(nv_n a_n^2)}\sum_{j=1}^{(r_n-l_n)/s_n}g(Y_{n,(j-1)s_n+(i-1)r_n+1})$.

\begin{corollary} \label{cor:disjoint_cond}
  Suppose that, in addition to (i) and (iii) of Corollary \ref{cor:sliding_cond}, the following conditions are satisfied:
\begin{itemize}

\item[(ii*)] For the sequences $l_n,r_n,s_n\in\N$ we have that $ l_n=\ord(r_n)$ is a multiple of $s_n$, $r_n=\ord(n)$, $p_n\to 0$, $r_n=\ord\big(\sqrt{n v_n}\big)$ and $(n/r_n)\beta_{n,l_n-s_n}^{X}\to 0$.

\item[(iv*)]
$ \displaystyle
E\bigg[\bigg(\sum_{j=1}^{l_n/s_n} \mathds{1}_{\left\{g(Y_{n,j}) \neq 0\right\}}\bigg)^2\bigg]  =\hbox{o}\big(r_n v_n a_n^2\big).$

\item [(v*)] $c^{(d)}$ as defined in (2.12) exists in $[0,\infty)$.

\item[(vi*)] $ \displaystyle E \bigg( \sum_{j=1}^{r_n/s_n} \big( g(Y_{n,(j-1)s_n+1})-Eg(Y_n)\big)^2 \ind{\{|\sum_{j=1}^{r_n/s_n}( g(Y_{n,(j-1)s_n+1})-Eg(Y_n))|>\varepsilon\sqrt{nv_n}a_n\}}\bigg)$\\
    \hspace*{1cm}$ \displaystyle =\ord\Big(\frac{r_np_n}n\Big)\qquad$ for all $\varepsilon>0$.

\end{itemize}	
 Then convergence \eqref{eq:vergl disjoint stat con} follows from Theorem A.1.
\end{corollary}

Condition (iv*) could be weakened to condition ($\Delta$) of Appendix A.

For the joint convergence
\begin{equation}
\label{eq:vergl est con}
\sqrt{nv_n}\begin{pmatrix}
T_n^\sharp(g)-E[T_n^\sharp(g)]\\
\frac{1}{nv_n}\sum_{i=1}^{n-s_n+1}\big(\mathds{1}_{\{X_{n,i}\neq 0\}}-v_n\big) \\
\end{pmatrix}
\xrightarrow{w} \begin{pmatrix}
T_1^{\sharp}\\
T_2
\end{pmatrix}\sim \mathcal{N}_2\left(0,\begin{pmatrix}	c^{(\sharp)}  & c^{(\sharp,v)} \\
c^{(\sharp,v)}  & c^{(v)}
\end{pmatrix}\right)
\end{equation}
as stated before Theorem 2.6 the conditions of Corollary \ref{cor:sliding_cond} and \ref{cor:disjoint_cond}, respectively, are sufficient, provided $c^{(\sharp,v)}$ and $c^{(v)}$ exist, with $\sharp\in\{s,d\}$.
The asymptotic normality of $\tilde{T}_n^s(g)=\sum_{i=1}^{n-s_n+1}g(Y_{n,i})/(s_na_n\sum_{i=1}^{n-s_n+1} \mathds{1}_{\{X_{n,i}\neq 0\}}$ centered at $\xi$ holds, if in addition the bias condition $$E[g(Y_{n})]/s_nv_na_n-\xi=\ord\big((nv_n)^{-1/2}\big)$$ is satisfied.
This is seen by the following continuous mapping argument:
\begin{align}
&\sqrt{nv_n}\left(\frac{T_n^s(g)}{(nv_n)^{-1}\sum_{i=1}^{n-s_n+1}\mathds{1}_{\{X_{n,i}\neq 0\}}}-\xi\right)\\
&=\sqrt{nv_n}\frac{T_n^s(g)-E[T_n^s(g)] - \xi(nv_n)^{-1}\sum_{i=1}^{n-s_n+1}(\mathds{1}_{\{X_{n,i}\neq 0\}}-v_n) + \frac{n-s_n+1}{n}(\frac{E[g(Y_{n,1})]}{s_nv_na_n} - \xi ) }{(nv_n)^{-1}\sum_{i=1}^{n-s_n+1}(\mathds{1}_{\{X_{n,i}\neq 0\}}-v_n)+\frac{n-s_n+1}{nv_n}v_n} \\
&=\frac{\sqrt{nv_n}(T_n^s(g)-E[T_n^s(g)]) - \xi(nv_n)^{-1/2}\sum_{i=1}^{n-s_n+1}(\mathds{1}_{\{X_{n,i}\neq 0\}}-v_n) +\ord(1) }{1+\ord_P(1)} \\
&\rightarrow T_1^{s}+\xi T_2,
\end{align}
where the last step follows from \eqref{eq:vergl est con}. The asymptotic normality of $\tilde{T}_n^d(g)$ follows along the same lines.

\section{Loewner order of covariance matrices}

The asymptotic variances of the sliding and disjoint blocks statistics for a single function $g$ are compared in Theorem 2.5. In the following a multivariate version of this result will be established. The asymptotic covariance matrices are compared w.r.t.\ the Loewner order, i.e. $A\leq_L B$ if and only if $B-A$ is positive semi-definite.

Fix some finite set $\GG$ of functions $g$ of the type considered in Subsection 2.1. If all functions in $\GG$ fulfill the conditions of Corollary \ref{cor:sliding_cond} and \ref{cor:disjoint_cond}, respectively, then
\begin{align}
\label{eq:vergl sliding stat hochdim con}
\big(\sqrt{nv_n}(T_n^s(g)-E T_n^s(g))\big)_{g\in\mathcal{G}}
&\rightarrow \mathcal{N}_{|\mathcal{G}|}(0,C^{(s)})\\
\label{eq:vergl disjoint stat hochdim con}
\big(\sqrt{nv_n}(T_n^d(g)-E T_n^d(g))\big)_{g\in\mathcal{G}}
& \rightarrow \mathcal{N}_{|\mathcal{G}|}(0,C^{(d)})
\end{align}
with $C^{(s)}=(c^{(s)}(g,h))_{g,h\in\mathcal{G}}$ and $C^{(d)}=(c^{(d)}(g,h))_{g,h\in\mathcal{G}}$, provided
\begin{align}
 c^{(s)}(g,h)&:=\lim_{n\to\infty}\frac{1}{r_nv_ns_n^2a_n^2}Cov\bigg(\sum_{i=1}^{r_n}g(Y_{n,i}), \sum_{i=1}^{r_n}h(Y_{n,i})\bigg),\\
c^{(d)}(g,h)&:=\lim_{n\to\infty} \frac{1}{r_nv_na_n^2} Cov\bigg(\sum_{i=1}^{r_n/s_n}g(Y_{n,is_n+1}),\sum_{i=1}^{r_n/s_n}h(Y_{n,is_n+1})\bigg),
\end{align}
exist in $[0,\infty)$ for all $g,h\in\GG$.

\begin{corollary}
	\label{cor:slidingvsdisjointloewner}
	If condition (A1) holds, $r_n/s_n \in\mathbb{N}$ for all $n\in\mathbb{N}$ and $|\mathcal{G}|<\infty$, then $C^{(s)}\leq_{L} C^{(d)}$.
\end{corollary}
Note that the assertion is equivalent to the statement that for all linear combinations $h$ of functions in $\GG$ the asymptotic variance of $T_n^s(h)$ is not greater than the corresponding asymptotic variance of $T_n^d(h)$.

\begin{proof}[Proof of Corollary \ref{cor:slidingvsdisjointloewner}]
		The inequality $C^s\leq_{L} C^d$ is equivalent to
	\begin{equation}
	\sum_{g,h\in\mathcal{G}} w_gw_h c^{(s)}(g,h)\leq \sum_{g,h\in\mathcal{G}} w_gw_h c^{(d)}(g,h)
	\end{equation}
	for all $(w_g)_{g\in\GG}\in\mathbb{R}^{|\mathcal{G}|}$.
	Obviously, for $\tilde f_w := \sum_{g\in\mathcal{G}} w_g g$,
	\begin{align}
	\sum_{g,h\in\mathcal{G}} w_gw_h c^{(s)}(g,h)
	 & =\lim_{n\to\infty} \frac{1}{r_nv_ns_n^2a_n^2}Cov\bigg(\sum_{g\in\mathcal{G}} w_g\sum_{k=1}^{r_n}g(Y_{n,k}),\sum_{h\in\mathcal{G}} w_h\sum_{k=1}^{r_n}h(Y_{n,k})\bigg)\\
	&=\lim_{n\to\infty} \frac{1}{r_nv_ns_n^2a_n^2}Var\bigg(\sum_{k=1}^{r_n}\tilde{f}_{w}(Y_{n,k})\bigg)=:c_w^{(s)}.
	\end{align}
	Likewise,
  $$ \sum_{g,h\in\mathcal{G}} w_gw_h c^{(s)}(g,h) = \lim_{n\to\infty} \frac{1}{r_nv_ns_n^2a_n^2}Var\bigg(\sum_{k=1}^{r_n/s_n}\tilde{f}_{w}(Y_{n,ks_n+1})\bigg)=:c_w^{(d)}.$$
		Since all conditions of Theorem 2.5 are satisfied for the single function $\tilde{f}_w$, this result yields $c_w^{(s)}\leq c_w^{(d)}$, i.e.\ the assertion.
\end{proof}

\section{Extremal index estimators using random thresholds }

In practice, when estimating extreme value parameters,  instead of a deterministic threshold $u_n$, usually a threshold $\hat{u}_n$ that depends on the observed time series is used. For instance, under suitable mixing conditions the $k_n$th largest order statistics $X_{n-k_n+1:n}$ of the observations satisfies $X_{n-k_n+1:n}/u_n\to 1$ in probability if $k_n=\ceil{n v_n}$ (cf.\ \cite{knezevic2020}, Lemma 2.2). In the following, we consider modified extremal index estimators where we replace the thresholds $u_n$ with random thresholds. The limit distribution of these estimators can be established with the theory from Section 2 and Appendix A using methods developed by \cite{knezevic2020}. To this end, though, we must consider more complicated empirical processes than in Section 3 and we need full process convergence instead of convergence of the finite dimensional distributions of these processes.
For the sake of brevity, we only consider the sliding blocks estimator, but the disjoint blocks estimator and the runs estimator can be analyzed in a similar fashion.

Recall that $M_{i,j}=\max(X_i,...,X_j)$ for $-\infty<i\leq j<\infty$.

The modified sliding blocks estimator with random threshold is defined as
\begin{align}
\hat{\theta}_{n,\hat{u}_n}^s:= \frac{\frac{1}{s_n}\sum_{i=1}^{n-s_n+1}\mathds{1}_{\{M_{i,i+s_n-1}> \hat{u}_n\}}}{\sum_{i=1}^{n-s_n+1}\mathds{1}_{\{X_i>\hat{u}_n\}}}.
\end{align}
The random threshold $\hat{u}_n$ is assumed consistent for $u_n$ in the sense that
\begin{equation}
\label{eq:Dn-to-1}
D_n:=\frac{\hat{u}_n}{u_n}\xrightarrow{P} 1.
\end{equation}

The basic idea of the  asymptotic analysis of $\hat{\theta}_{n,\hat{u}_n}^s$ is to amend the empirical process $(\bar Z_n(g),\bar Z_n(h))$ used in the proof of Theorem 3.1 by an additional parameter $d\in[1-\eps,1+\eps]$ (for some $\eps>0$) that later on is replaced with $D_n$.
We thus modify some of the conditions used in Section 3 as follows:
\begin{enumerate}
	\item[\bf($\theta$PR)] There exist $\eps>0$ and, for  all $n\in\mathbb{N}$ and $k\in\mathbb{N} $, $e_n(k)$ such that
	\begin{equation}
	e_n(k)\geq P(X_k>(1-\epsilon)u_n \mid X_0>(1-\epsilon)u_n)
	\end{equation}
	and $\lim_{n\to\infty}\sum_{k=1}^{r_n}e_n(k)=\sum_{k=1}^{\infty}\lim_{n\to\infty}e_n(k) <\infty$.
	
	\item[\bf(B$\mathbf{_b}$R)] For all sequences $d_n\to 1$,
	$ \displaystyle
	\frac{P\{M_{1,s_n}> d_nu_n\}}{s_nP\{X_0>d_nu_n\}}  - \theta = \ord\big((nv_n)^{-1/2}\big).
	$
\end{enumerate}
In addition, we assume that the positive part $(X^+_t)_{t\in\mathbb{Z}}:=(X_t\mathds{1}_{\{X_t\geq0\}})_{t\in\mathbb{Z}}$ of the time series is regular varying, that is all finite dimension distributions are multivariate regularly varying. According to \cite{basrak2009}, Theorem 2.1, this assumption is equivalent to the existence of a tail process $(W_t)_{t\in\Z}$ such that
\begin{equation}  \label{eq:tailrpocdef}
  P\Big( \Big(\frac{X_i}u\Big)_{s\le i\le t}\in \cdot \,\Big|\, X_0>u\Big)\; \stackrel{w}{\to} \; P\big((W_i)_{s\le i\le t}\in \cdot \big)
\end{equation}
as $u\to\infty$ for all $-\infty<s\le t<\infty$. Then $W_0$ is independent of the so-called (tail) spectral process $(\Theta_t)_{t\in\Z}:=(W_t/W_0)_{t\in\Z}$ and $P\{W_0>t\}=t^{-\alpha}$ for all $t>1$ and some $\alpha>0$, the so-called index of regular variation. We thus assume
\begin{enumerate}
	\item[\bf(R)] $(X^+_t)_{t\in\mathbb{Z}}$ is regular varying with tail process $(W_t)_{t\in\Z}$, spectral process $(\Theta_t)_{t\in\Z}$ and index $\alpha$.
\end{enumerate}
Observe that if  ($\theta$PR) and (R) are satisfied then the following generalization of ($\theta$PR) holds, too: one has eventually for all $c,d\in[1-\eps,1+\eps]$
    \begin{align}
	 P(X_k>cu_n \mid X_1>du_n)
  & \le P(X_k>(1-\epsilon)u_n \mid X_1>(1-\epsilon)u_n)\cdot \frac{P\{X_1>(1-\epsilon)u_n\}}{P\{X_1>(1+\eps) u_n\}} \\
  &\le 2 \Big(\frac{1-\eps}{1+\eps}\Big)^{-\alpha} e_n(k) =:\tilde e_n(k)
	\end{align}
with $\lim_{n\to\infty}\sum_{k=1}^{r_n}\tilde e_n(k)=\sum_{k=1}^{\infty}\lim_{n\to\infty}\tilde e_n(k) <\infty$.

The following result shows that the sliding blocks estimator with random thresholds has the same limit distribution as the estimators with deterministic thresholds.
\begin{theorem}
	\label{thm:extremal sliding random}
	If the conditions ($\theta$1), ($\theta$PR), (B$_b$R), (R) and \eqref{eq:Dn-to-1} are satisfied, then
	\begin{equation}
	\sqrt{nv_n}(\hat{\theta}_{n,\hat{u}_n}^S-\theta)\xrightarrow{w} \mathcal{N}(0,\theta( \theta c-1))
	\end{equation}
  with $c$ defined in condition ($\theta$2).
\end{theorem}
Due to condition ($\theta$P) and (R), $c$ can be represented in terms of the tail process $(W_t)_{t\in\mathbb{Z}}$ by
\begin{align}
c
& = 1+ \lim_{n\to\infty} \sum_{k=1}^{r_n-1}\Big(1-\frac{k}{r_n}\Big)\big(P(X_{k}>u_n|X_0>u_n)+P(X_{0}>u_n|X_{-k}>u_n)\big)\\
&=1+2\sum_{k=1}^\infty P\{W_{k}>1\}
\end{align}
(see also Section 3).

In the proof, we consider  the class $\mathcal{G}=\{g_d,h_d: d\in [1-\epsilon,1+\epsilon] \}$ of bounded functions
\begin{align}
g_d(x_1,\ldots,x_s)&:=\mathds{1}_{\{\max_{1\le i\le s} x_i> d\}},\\
h_d(x_1,\ldots,x_s)&:=\mathds{1}_{\{x_1> d\}}.
\end{align}
We choose $X_{n,i}:=(X_i/u_n)\ind{\{X_i>u_n\}}$ and the normalizing sequence $b_n(g_d)=\sqrt{nv_n/p_n^s}s_n$ and $b_n(h_d)=\sqrt{nv_n/p_n^s}$ with $p_n^s=P\{M_{1,r_n+s_n-1}>(1-\epsilon)u_n\}=(1-\eps)^{-\alpha} r_nv_n\theta(1+\ord(1))$ by (3.1) and regular variation. Define $\bar{Z}_n$ as in (2.1), i.e.
\begin{align}
\bar{Z}_n(g_d)& =\frac{1}{\sqrt{nv_n}s_n}\sum_{i=1}^{n-s_n+1}(\ind{\{M_{i,i+s_n-1}>du_n\}}-P\{M_{1,s_n}>du_n\})\\
\bar{Z}_n(h_d)& =\frac{1}{\sqrt{nv_n}}\sum_{i=1}^{n-s_n+1}(\ind{\{X_i>du_n\}}-P\{X_0>du_n\}).
\end{align}

In a first step we prove the asymptotic normality of $(\bar{Z}_n(f))_{f\in\mathcal{G}}$. There we use the notation
$U^*_{s,t}:=\sup_{s\le i\le t,i\in\Z} U_i$ for a process $(U_i)_{i\in\Z}$ and $-\infty\le s\le t\le\infty$.
\begin{proposition}
	\label{prop:asymp sliding random nom and denom}
	If the conditions ($\theta$1), ($\theta$PR) and (R) are satisfied, then
	\begin{equation}
	(\bar{Z}_n(f))_{f\in\mathcal{G}} \xrightarrow{w} (Z(f))_{f\in\mathcal{G}},
	\end{equation}
	where $Z$ is a centered Gaussian process with $Var(Z(g_1))=\theta$, $Var(Z(h_1))=c$ and $Cov(Z(g_1),$ $Z(h_1))=1$.
\end{proposition}

\begin{proof}
	Since the functions $g_d$ and $h_d$ are bounded, we can apply Theorem 2.3. The Conditions (A1), (A2) and (MX) easily follow from $(\theta1)$ and (R), and Condition (D0) is obvious.
	Condition (2.5) for $g_{1-\epsilon}$, and hence for all $g_d$, $d\in[1-\eps,1+\eps]$,  follows from ($\theta$PR) in the same way as in the proof of Proposition B.3.  One may argue similarly to verify (2.5) for $h_d$, $d\in[1-\eps,1+\eps]$.
	
	The convergence of the covariance matrix is established in Lemma \ref{lemma:covariance.extremal.random}. In particular, for $d=1$ one obtains the same covariances as in Proposition B.3, since $P\{W_0>1\}=1$.

	It remains to check the conditions for convergence of the process, i.e.\ (D1) and (D3).
	
	The functionals $(g_d)_{d\in [1-\epsilon,1+\epsilon]}$ are linearly ordered: $g_d \geq g_{d'}$ for $d\leq d'$. 
	Hence, $(g_d)_{d\in [1-\epsilon,1+\epsilon]}$ forms a VC(2)-class and (D3) is satisfied for $(g_d)_{d\in[1-\epsilon,1+\epsilon]}$ (cf.\ \cite{drees2010}, Remark 2.11). The same argument applies to $(h_d)_{d\in[1-\epsilon,1+\epsilon]}$. Therefore, Condition (D3) is satisfied.
	
	Note that the remaining Condition (D1) may be verified for the functions $(g_d)_{d\in (1-\epsilon,1+\epsilon)}$ and $(h_d)_{d\in [1-\epsilon,1+\epsilon]}$ separately. In the remaining parts we only check it for $(g_d)_{d\in [1-\epsilon,1+\epsilon]}$, as the assertion for $h_d$ follows along the same lines.
	
	The family $(g_d)_{d\in [1-\epsilon,1+\epsilon]}$ is totally bounded with respect to the metric $\rho(g_{d},g_{d'}):=|d^{-\alpha}-d'^{-\alpha}|$. Lemma \ref{lemma:covariance.extremal.random} (iii) yields
  \begin{align}
    \frac{1}{r_ns_n^2v_n}& E\bigg[ \sum_{i=1}^{r_n}g_d(Y_{n,i})\cdot\sum_{i=1}^{r_n}g_{d'}(Y_{n,i})\bigg]\\
      & = \frac{1}{r_ns_n^2v_n}Cov\bigg( \sum_{i=1}^{r_n}g_d(Y_{n,i}),\sum_{i=1}^{r_n}g_{d'}(Y_{n,i})\bigg) +\Ord(r_nv_n)\\
       & \to \frac 12\Big( P\big\{W_{1,\infty}^*\le 1, W_{-\infty,\infty}^*>d'/d\big\}d^{-\alpha} + P\big\{W_{1,\infty}^*\le 1, W_{-\infty,\infty}^*>d/d'\big\}d'^{-\alpha}\Big)\\
       & =: D(d,d'). \label{eq:covlim}
  \end{align}
  Because
  \begin{align}
    P\Big\{W_{1,\infty}^*\le 1,W_{-\infty,\infty}^*>c\Big\}
    &= P\Big\{ W_0\le \frac 1{\Theta_{1,\infty}^*}, W_0>\frac c{\Theta_{-\infty,\infty}^*}\Big\}\\
    & = \int\Big(\Big(\max\Big(\frac ct,1\Big)\Big)^{-\alpha}-\Big(\frac 1s\Big)^{-\alpha}\Big)^+\; P^{(\Theta_{1,\infty}^*,\Theta_{-\infty,\infty}^*)}(ds,dt)
  \end{align}
  for all $c>0$, the limit in \eqref{eq:covlim} is a continuous function of $(d,d')\in[1-\eps,1+\eps]^2$. Moreover, the left-hand side of \eqref{eq:covlim} is monotone in $d$ and $d'$. Hence,  convergence \eqref{eq:covlim} holds uniformly on $[1-\eps,1+\eps]^2$.

  Since $W_{-\infty,\infty}^*>1$ almost surely, we may conclude, uniformly for $1-\eps\le d\le d'\le 1+\eps$,
  \begin{align}
	\frac{1}{r_ns_n^2v_n}&E\bigg[\bigg( \sum_{i=1}^{r_n}(g_d(Y_{n,i})-g_{d'}(Y_{n,i}))\bigg)^2 \bigg]\\
    & \to D(d,d)+D(d',d')-2D(d,d')\\
    & = P\big\{W_{1,\infty}^*\le 1\}(d^{-\alpha}+d'^{-\alpha})-P\big\{W_{1,\infty}^*\le 1, W_{-\infty,\infty}^*>d'/d\big\}d^{-\alpha}\\
    & \hspace{1cm}  - P\big\{W_{1,\infty}^*\le 1, W_{-\infty,\infty}^*>d/d'\big\}d'^{-\alpha}\\
    & = P\big\{W_{1,\infty}^*\le 1\big\} d^{-\alpha}-P\big\{W_{1,\infty}^*\le 1, W_{-\infty,\infty}^*>d'/d\big\}d^{-\alpha}\\
    & = P\big\{W_{1,\infty}^*\le 1, W_{-\infty,\infty}^*\le d'/d\big\}d^{-\alpha}\\
    & \le P\big\{ W_0\le d'/d\big\}d^{-\alpha}=\big(1-(d'/d)^{-\alpha}\big)d^{-\alpha}=\rho(g_d,g_{d'}).
  \end{align}
 	Thus,
	\begin{align}
	 \limsup_{n\to\infty} \sup_{d,d'\in[1-\eps,1+\eps],\rho(g_d,g_{d'})<\delta} \frac{1}{r_ns_n^2v_n}E\bigg[\bigg( \sum_{i=1}^{r_n}(g_d(Y_{n,i})-g_{d'}(Y_{n,i}))\bigg)^2 \bigg]\le \delta,
	\end{align}
	i.e.\ Condition (D1) is satisfied.

	Now Theorem 2.3 yields the assertion.
\end{proof}

It remains to show that the standardized covariances of $\bar Z_n$ converge.
\begin{lemma}
	\label{lemma:covariance.extremal.random}
	If the conditions ($\theta$1), ($\theta$PR) and (R) are met, then the following three limits exists for all $c,d\in[1-\epsilon,1+\epsilon]$:
	\begin{enumerate}
		\item[(i)]
		\begin{align}
		\lim_{n\to\infty}&\frac{1}{r_nv_n}Cov\bigg(\sum_{i=1}^{r_n}\mathds{1}_{\{X_i>cu_n\}}, \sum_{j=1}^{r_n}\mathds{1}_{\{X_j>du_n\}}  \bigg) \\
		&= \sum_{k=1}^{\infty}P\Big\{W_k>\frac{c}{d}\Big\}d^{-\alpha}+ \sum_{k=1}^{\infty}P\Big\{W_k>\frac{d}{c}\Big\}c^{-\alpha} + (\max(c,d))^{-\alpha}
		\end{align}
		
		\item[(ii)] \begin{align}
		&\lim_{n\to\infty}\frac{1}{r_n s_n v_n}Cov\bigg(\sum_{i=1}^{r_n}\mathds{1}_{\{M_{i,i+s_n-1}> cu_n\}},\sum_{j=1}^{r_n}\mathds{1}_{\{X_j>du_n\}}\bigg) = P\Big\{W^*_{-\infty,\infty}>\frac{c}{d}\Big\}d^{-\alpha}
		\end{align}
		
		\item[(iii)] \begin{align}
		\lim_{n\to\infty}&\frac{1}{r_ns_n^2v_n}Cov\bigg(\sum_{i=1}^{r_n}\mathds{1}_{\{M_{i,i+s_n-1}> cu_n\}},\sum_{j=1}^{r_n}\mathds{1}_{\{M_{j,j+s_n-1}> du_n\}}\bigg)\\
		&=\frac{1}{2} \bigg(P\Big\{W^*_{1,\infty}\leq 1, W^*_{-\infty,\infty}>\frac{d}{c}\Big\}c^{-\alpha}+ P\Big\{W^*_{1,\infty}\leq 1, W^*_{-\infty,\infty}>\frac{c}{d}\Big\}d^{-\alpha}\bigg).
		\end{align}
	\end{enumerate}
\end{lemma}

 The following technical lemma is used in the proof of Lemma \ref{lemma:covariance.extremal.random}.

\begin{lemma}
	\label{lemma:tailprocess.wachsender.ausschnitt}
	Suppose the conditions ($\theta$PR) and (R) are satisfied. Then, for all sequences $t_n,\tilde t_n,t^*_n\to\infty$, $t_n,\tilde t_n,t^*_n\le r_n$ and all $c,d,d^*\in[1-\eps,1+\eps]$,
	\begin{align}
	P\big(M_{-t_n,\tilde t_n}>d u_n, M_{1,t_n^*}\le d^*u_n\mid X_0>cu_n\big) & \to P\big\{ W^*_{-\infty,\infty}>d/c,W^*_{1,\infty}\le d^*/c\big\}, \label{eq:probmaxconv1}\\
    P\big(M_{-t_n,\tilde t_n}>d u_n\mid X_0>cu_n\big) & \to P\big\{W^*_{-\infty,\infty}>d/c\big\}. \label{eq:probmaxconv2}
	\end{align}
\end{lemma}

\begin{proof}
	First note that under ($\theta$PR) and (R), the tail process will finally not exceed $(1-\eps)/(1+\eps)$, i.e., $\lim_{l\to\infty}P\big\{\sup_{|t|>l} W_t>(1-\eps)/(1+\eps)\big\}= 0$. To see this, check that for all $1\le l\le m$
    \begin{align}
  P\Big\{ W_{l,m}^*>\frac{1-\eps}{1+\eps}\Big\} & = \lim_{n\to\infty} P\big( M_{l,m}>(1-\eps)u_n\mid X_0>(1+\eps)u_n\big)\\
    & \le \limsup_{n\to\infty} \sum_{j=l}^m P(X_j>(1-\eps)u_n\mid X_0>(1+\eps)u_n)\\
    & \le  \sum_{j=l}^\infty \lim_{n\to\infty} e_n(j) \Big(\frac{1-\eps}{1+\eps}\Big)^{-\alpha}.
    \end{align}
    By monotone convergence, one may conclude $\lim_{l\to\infty}P\big\{W_{l,\infty}^*>(1-\eps)/(1+\eps)\big\}=0$. The proof of
    $\lim_{l\to\infty}P\big\{ W_{-\infty,-l}^*>(1-\eps)/(1+\eps)\big\}=0$ is similar.

    Hence, for any fixed $\eta>0$, there exists $m_\eta\in\N$ such that for all $m\ge m_\eta$
    $$ \big|P\big\{W^*_{-\infty,\infty}>d/c,W^*_{1,\infty}\le d^*/c\big\} - P\big\{W^*_{-m,m}>d/c,W^*_{1,m}\le d^*/c \big\}\big| < \eta/3.
    $$
    Moreover,
    \begin{align}
      \Big| P & \big(M_{-t_n,\tilde t_n}>d u_n, M_{1,t_n^*}\le d^*u_n\mid X_0>cu_n\big)  \\ & \hspace{1cm} -P\big(M_{-m,m}>d u_n, M_{1,m}\le d^*u_n\mid X_0>(1-\eps)u_n\big)\Big| \\
      & \le \sum_{k=m+1}^{\max(\tilde t_n,t_n^*)} P(X_k>(1-\eps)u_n\mid X_0>(1-\eps)u_n)\cdot \frac{P\{X_{0}>(1-\eps)u_n\}}{P\{X_0>cu_n\}}\\
      & \hspace{1cm} +
       \sum_{k=m+1}^{t_n} P(X_0>cu_n\mid X_{-k}>(1-\eps)u_n)\cdot \frac{P\{X_{-k}>(1-\eps)u_n\}}{P\{X_0>cu_n\}}\\
      & \le 3\Big(\frac{1-\eps}{1+\eps}\Big)^{-\alpha} \sum_{k=m+1}^{\max(t_n,\tilde t_n,t_n^*)} e_n(k)\\
      & \le \frac{\eta}3
    \end{align}
    for sufficiently large $m$ and $n$. Therefore, one has eventually
    \begin{align}
      \Big|&P\big(  M_{-t_n,\tilde t_n}>d u_n, M_{1,t_n^*}\le d^*u_n\mid X_0>cu_n\big)- P\big\{W^*_{-\infty,\infty}>d/c, W^*_{1,\infty}\le d^*/c\big\}\Big|\\
      & < \big|P\big(M_{-m,m}>d u_n, M_{1,m}\le d^*u_n\mid X_0>cu_n\big)- P\big\{W^*_{-m,m}>d/c,W^*_{1,m}\le d^*/c \big\}\big| + \frac 23\eta\\
      & < \eta
    \end{align}
    by definition \eqref{eq:tailrpocdef} of the tail process. Since $\eta>0$ is arbitrary, this proves \eqref{eq:probmaxconv1}. The second assertion can be established by similar arguments.
\end{proof}

\begin{proof}[Proof of Lemma \ref{lemma:covariance.extremal.random}]
   To prove assertion (i), first note that by regular variation $P\{X_0>du_n\} =d^{-\alpha} v_n(1+\ord(1))$ for all $d>0$.
   Hence, by stationarity,
	\begin{align}
	\frac{1}{r_nv_n} & Cov\bigg(\sum_{i=1}^{r_n}\mathds{1}_{\{X_i>cu_n\}},\sum_{j=1}^{r_n}\mathds{1}_{\{X_j>du_n\}}  \bigg)\\
	& = \frac{1}{r_nv_n} \sum_{i=1}^{r_n} \sum_{j=1}^{r_n} P\{X_i>cu_n, X_j>du_n\} +\hbox{O}(r_nv_n)\\
	& = \sum_{k=1}^{r_n-1} \Big(1-\frac{k}{r_n}\Big)  P(X_k>cu_n| X_0>du_n)\frac{P\{X_0>du_n\}}{v_n}\\
	&\hspace{1cm} + \sum_{k=1}^{r_n-1} \Big(1-\frac{k}{r_n}\Big)  P(X_k>du_n| X_0>cu_n)\frac{P\{X_0>cu_n\}}{v_n}\\
	&\hspace{1cm} + \frac{1}{v_n}P\{X_0>\max(c,d)u_n\} +\hbox{O}(r_nv_n)\\
	&\rightarrow \sum_{k=1}^{\infty}P\Big\{W_k>\frac{c}{d}\Big\}d^{-\alpha}+ \sum_{k=1}^{\infty}P\Big\{W_k>\frac{d}{c}\Big\}c^{-\alpha} + \big(\max(c,d)\big)^{-\alpha}.
	\end{align}
	In the last step we have used regular variation and Pratt's lemma, that can be applied due to Condition ($\theta$PR) (see the  discussion given below Condition (R)).
	
	Next note that the following generalizations to the equations (B.10) and (B.11) hold for all $c,d\in[1-\epsilon,1+\epsilon]$:
	\begin{align}
	&\sum_{j=\ell+s_n+1}^{r_n}  P\big\{M_{\ell+1,\ell+s_n}>cu_n, X_j>du_n\big\}\\
	& \le \sum_{j=\ell+s_n+1}^{r_n}  P\big\{M_{\ell+1,\ell+s_n}>(1-\eps)u_n, X_j>(1-\eps)u_n\big\}\\
	& = \ord(s_nv_n)  \label{eq:summaxbound1}
	\end{align}
 and	
	\begin{align}
	\sum_{j=1}^{\ell} P\big\{M_{\ell+1,\ell+s_n}>cu_n, X_j>du_n\big\}	
	= \ord(s_nv_n)  \label{eq:summaxbound2}
	\end{align}
	uniformly for $1\le \ell\le r_n-s_n$.	
	It follows for the left hand side  in (ii)
	\begin{align}
	\frac{1}{r_n s_n v_n} & Cov\bigg(\sum_{i=1}^{r_n}\mathds{1}_{\{M_{i,i+s_n-1}> cu_n\}},\sum_{j=1}^{r_n}\mathds{1}_{\{X_j>du_n\}}\bigg)\\
	& = \frac{1}{r_n s_n v_n} \sum_{i=1}^{r_n} \sum_{j=1}^{r_n} P\{M_{i,i+s_n-1}> cu_n, X_j>du_n\} + \Ord(r_nv_n) \\
	& = \frac{1}{r_n s_n v_n} \sum_{i=1}^{r_n} \sum_{j=i}^{\min(i+s_n-1, r_n)} P\{M_{i,i+s_n-1}> cu_n,X_j>du_n\} + \ord(1)\\
	& = \frac{1}{s_n} \sum_{k=-s_n+1}^{0} \Big(1-\frac{|k|}{r_n}\Big) P(M_{k,k+s_n-1}> cu_n \mid X_0>du_n) \frac{P\{X_0>du_n\}}{v_n}+ \ord(1).
	\end{align}
	Moreover, for any sequence $t_n\to \infty$, $t_n=\hbox{o}(s_n)$
	\begin{align}
	\frac{1}{s_n} & \sum_{k=-s_n+t_n+1}^{-t_n} P(M_{k,k+s_n-1}> cu_n \mid X_0>du_n)\\
	&\leq \frac{s_n-2t_n}{s_n} P(M_{-s_n,s_n}> cu_n \mid X_0>du_n)\\
    & \rightarrow P\Big\{W^*_{-\infty,\infty}> \frac{c}{d}\Big\},
	\end{align}
	where \eqref{eq:probmaxconv2} was applied in the last step. Likewise, for sufficiently large $n$,
	\begin{align}
	\frac{1}{s_n} & \sum_{k=-s_n+t_n+1}^{-t_n} P(M_{k,k+s_n-1}> cu_n\mid X_0>du_n)\\
	&\geq \frac{s_n-2t_n}{s_n} P(M_{-t_n,t_n}> cu_n \mid X_0>du_n)\\
    & \rightarrow P\Big\{W^*_{-\infty,\infty}> \frac{c}{d}\Big\}.
	\end{align}
	Thus,
	\begin{align}
	&\frac{1}{s_n} \sum_{k=-s_n+1}^{0} \Big(1-\frac{|k|}{r_n}\Big) P(M_{k,k+s_n-1}> cu_n\mid X_0>d u_n) \frac{P\{X_0>du_n\}}{v_n}+ \ord(1)\\
	&= \frac{1}{s_n} \sum_{k=-s_n+t_n+1}^{-t_n} P(M_{k,k+s_n-1}> cu_n \mid X_0>du_n)\frac{P\{X_0>du_n\}}{v_n} \\
    & \hspace{8cm} +\hbox{O}\left(\frac{t_n}{s_n}+\frac{s_n}{r_n}\right)+\hbox{o}(1)\\
	&\rightarrow P\Big\{W^*_{-\infty,\infty}> \frac{c}{d}\Big\}d^{-\alpha}.
	\end{align}
	
	Finally, we turn to (iii). The arguments are similar to the arguments used in the proof of assertion (ii) and in the proof of Lemma B.5 (iv).
  By stationarity,	
	\begin{align}
	Cov\bigg(& \sum_{i=1}^{r_n}  \mathds{1}_{\{M_{i,i+s_n-1}> cu_n\}},\sum_{i=1}^{r_n}  \mathds{1}_{\{M_{i,i+s_n-1}> du_n\}}\bigg)\\
	& = \sum_{i=1}^{r_n} \sum_{j=1}^{r_n} P\{M_{i,i+s_n-1}> cu_n,M_{j,j+s_n-1}> du_n\}  + \Ord\big((r_ns_nv_n)^2\big)\\
	& = \sum_{i=1}^{r_n} \sum_{j=i}^{r_n} P\{M_{i,i+s_n-1}> cu_n,M_{j,j+s_n-1}> du_n\} \\
	& \hspace{1cm} +\sum_{i=1}^{r_n} \sum_{j=i}^{r_n} P\{M_{i,i+s_n-1}> du_n,M_{j,j+s_n-1}> cu_n\} + \ord(r_ns_n^2v_n).
   \label{eq:covsplitup}
	\end{align}
	For all $c,d\in[1-\epsilon,1+\epsilon]$ one can decompose the first term as follows:
	\begin{align}
	\sum_{i=1}^{r_n} & \sum_{j=i}^{r_n} P\{M_{i,i+s_n-1}> cu_n,M_{j,j+s_n-1}> du_n\}\\
	&=  \sum_{i=1}^{r_n-3s_n} \sum_{j=i}^{i+s_n-1} P\{M_{i,i+s_n-1}> cu_n,M_{j,j+s_n-1}> du_n\}\\
	& \hspace{1.5cm}+ \sum_{i=r_n-3s_n+1}^{r_n} \sum_{j=i}^{r_n} P\{M_{i,i+s_n-1}> cu_n,M_{j,j+s_n-1}> du_n\}\\
	& \hspace{1.5cm}+\sum_{i=1}^{r_n-3s_n} \sum_{j=i+s_n}^{r_n-s_n} P\{M_{i,i+s_n-1}> cu_n,M_{j,j+s_n-1}> du_n\}\\
	& \hspace{1.5cm}+\sum_{i=1}^{r_n-3s_n} \sum_{j=r_n-s_n+1}^{r_n} P\{M_{i,i+s_n-1}> cu_n,M_{j,j+s_n-1}> du_n\}\\
	& =: I+II+III+IV.
	\end{align}
	
	As in the proof of Lemma B.5, one can show  that the terms II, III, and IV are of smaller order than $r_ns_n^2v_n$.
	We next show that
\begin{align}
	\frac{I}{r_ns_n^2v_n}
    & =\frac{1+\ord(1)}{s_n^2v_n}\sum_{k=1}^{s_n} P\{M_{1,s_n}> cu_n,M_{k,k+s_n-1}> du_n\}\\
    & \to \frac{1}{2}P\Big\{W^*_{1,\infty}\leq 1, W^*_{-\infty,\infty}>\frac{d}{c}\Big\}c^{-\alpha}. \label{eq:lastconv}
  \end{align}
	Distinguish according to the last exceedance in $\{1,\ldots, s_n\}$ to conclude
	\begin{align}
	\sum_{k=1}^{s_n} & P\{M_{1,s_n}> cu_n,M_{k,k+s_n-1}> du_n\}\\
	& = \sum_{k=1}^{s_n} \sum_{i=1}^{s_n} P\{X_i>cu_n,M_{i+1,s_n}\le cu_n,M_{k,k+s_n-1}>du_n\}\\
	& = \sum_{k=1}^{s_n} \sum_{i=k}^{s_n} P\{X_i>cu_n,M_{i+1,s_n}\le cu_n,M_{k,k+s_n-1}>du_n\}\\
	& \hspace{1cm} +\Ord\bigg(\sum_{k=1}^{s_n} \sum_{i=1}^{k-1} P\{X_i>cu_n,M_{k,k+s_n-1}>du_n\}\bigg)\\
    & = \sum_{k=1}^{s_n} \sum_{i=k}^{s_n} P\{X_0>cu_n,M_{1,s_n-i}\le cu_n,M_{k-i,k-i+s_n-1}>du_n\}+\ord(s_n^2 v_n)\\
	& = \sum_{k=1}^{s_n} \sum_{j=k-s_n}^{0} P\{X_0>cu_n,M_{1,s_n+j-k}\le cu_n,M_{j,j+s_n-1}>du_n\}
	+\ord(s_n^2v_n)\\
	& = \sum_{j=1-s_n}^{0} \sum_{k=1}^{s_n+j} P\{X_0>cu_n,M_{1,s_n+j-k}\le cu_n,M_{j,j+s_n-1}>du_n\} +\ord(s_n^2v_n).
	\end{align}
	where in the third step we have employed \eqref{eq:summaxbound2}.
	For any sequence $t_n\to \infty$, $t_n=\ord(s_n)$, this last sum can eventually be bound from below by
	\begin{align}
	\sum_{j=-s_n+t_n}^{-t_n}& \sum_{k=1}^{s_n+j} P\{X_0>cu_n,M_{1,s_n+j-k}\le cu_n,M_{j,j+s_n-1}>du_n\} +\Ord(t_ns_nv_n)\\
	&\geq \sum_{j=-s_n+t_n}^{-t_n} (s_n+j-1) P\{X_0>cu_n,M_{1,s_n}\le cu_n,M_{-t_n,t_n-1}>du_n\}+ \ord(s_n^2v_n)\\
	&= \frac{s_n^2v_n}{2}\frac{P\{X_0>cu_n\}}{v_n}P\Big\{W^*_{1,\infty}\leq 1, W^*_{-\infty,\infty}>\frac{d}{c}\Big\}(1+\ord(1))
	\end{align}
	by \eqref{eq:probmaxconv1}.
	Similarly, the sum has the upper bound
	\begin{align}
	\sum_{j=t_n-s_n+1}^{0} & \sum_{k=1}^{s_n+j-t_n} P\{X_0>cu_n,M_{1,s_n+j-k}\le cu_n,M_{j,j+s_n-1}>du_n\} + \Ord(t_ns_nv_n)\\
	&\leq \sum_{j=-s_n}^{0} (s_n+j-t_n) P\{X_0>cu_n,M_{1,t_n}\le cu_n,M_{-s_n,s_n}>du_n\}+ \ord(s_n^2v_n)\\
	&= \frac{s_n^2v_n}{2}c^{-\alpha}P\Big\{W^*_{1,\infty}\leq 1, W^*_{-\infty,\infty}>\frac{d}{c}\Big\}(1+\ord(1)).
	\end{align}
	Hence convergence \eqref{eq:lastconv} follows, which gives the asymptotic behavior of the first term in  \eqref{eq:covsplitup}. Interchanging the role of $c$ and $d$ yields the analogous result for the second term, which concludes the proof of (iii).
\end{proof}

Finally, we prove the asymptotic normality of the sliding blocks estimator based on the exceedances over the random threshold $\hat u_n=D_n u_n$.
\begin{proof}[Proof of Theorem \ref{thm:extremal sliding random}]
	
	By Proposition \ref{prop:asymp sliding random nom and denom}
	\begin{equation}
	(\bar{Z}_n(g_{d}),\bar{Z}_n(h_{d}))_{d\in[1-\epsilon,1+\epsilon]}\xrightarrow{w} (Z(g_{d}),Z(h_d))_{d\in[1-\epsilon,1+\epsilon]}.
	\end{equation}
	Using \eqref{eq:Dn-to-1} and Slutsky's Lemma, we may conclude
	\begin{align} \label{eq:slutskyjointconv}
	\big((\bar{Z}_n(g_{d}),\bar{Z}_n(h_{d}))_{d\in[1-\epsilon,1+\epsilon]},D_n\big)\xrightarrow{w} \big((Z(g_{d}),Z(h_d))_{d\in[1-\epsilon,1+\epsilon]},1\big).
	\end{align}
	By Skorohod's representation theorem there exist versions of these processes which converge almost surely. According to (\cite{vanderVaart1996}, Addendum 1.5.8), under the conditions (D0), (D1) and (D3) established in the proof of Proposition \ref{prop:asymp sliding random nom and denom}, the limit process $(Z(g_d))_{d\in[1-\epsilon,1+\epsilon]}$  has almost surely continuous sample paths w.r.t.\ the metric $\rho_g$. Analogously $(Z(h_d))_{d\in[1-\epsilon,1+\epsilon]}$ has almost surely continuous sample paths.
	Therefore, in view of the almost sure version of \eqref{eq:slutskyjointconv},
	\begin{align} \label{eq:zngdconv}
	|\bar Z_n(g_{D_n})-Z(g_1)|&\leq|\bar Z_n(g_{D_n})-Z(g_{D_n})|+|Z(g_{D_n})-Z(g_1)| \\
	&\leq \sup_{d\in[1-\epsilon,1+\epsilon]}|\bar Z_n(g_{d})-Z(g_d)|+|Z(g_{D_n})-Z(g_1)|\rightarrow 0
	\end{align}
	almost surely, and likewise
	\begin{align} \label{eq:znhdconv}
	|\bar Z_n(h_{D_n})-Z(h_1)|\rightarrow 0
	\end{align}
	almost surely.

    Now, for $d\in [1-\epsilon,1+\epsilon]$, define
	\begin{equation}
	\theta_{n,d}^s:= \frac{\frac{1}{s_n}\sum_{i=1}^{n-s_n+1}g_d(Y_{n,i})}{\sum_{i=1}^{n-s_n+1}h_d(Y_{n,i})},
	\end{equation}
so that $\hat{\theta}_{n,\hat{u}_n}^s=\theta_{n,D_n}^s$ if $D_n\in[1-\eps,1+\eps]$, which holds with probability tending to $1$. Let
$$\theta_n(d):=\frac{P(M_{1,s_n}> du_n)}{s_nP(X_0>du_n)}.$$
   Then, for any sequence $d_n\to 1$,
	\begin{align}
	\sqrt{nv_n}&(\theta_{n,d_n}^s-\theta_n(d_n))\\
	&=\sqrt{nv_n}\left(\frac{\frac{1}{\sqrt{nv_n}}\bar{Z}_n(g_{d_n}) +\frac{n-s_n+1}{nv_ns_n}P(M_{1,s_n}>d_nu_n)}{\frac{1}{\sqrt{nv_n}}\bar{Z}_n(h_{d_n})+\frac{n-s_n+1}{nv_n}P(X_0>d_nu_n)}-\theta_n(d_n)\right)\\
	&=\frac{\bar{Z}_n(g_{d_n}) - \theta_n(d_n)  \bar{Z}_n(h_{d_n}) }{\frac{1}{\sqrt{nv_n}}\bar{Z}_n(h_{d_n})+\frac{n-s_n+1}{nv_n}P(X_0>d_nu_n)}.
	\end{align}
  The denominator tends  to 1 by \eqref{eq:slutskyjointconv} and regular variation.
	Since $\theta_n(D_n)\to \theta$ by the bias condition (B$_b$R), from $D_n\to 1$, \eqref{eq:zngdconv} and \eqref{eq:znhdconv} we may conclude
  $$ \sqrt{nv_n}(\theta_{n,D_n}^s-\theta_n(D_n)) \to Z(g_1)-\theta Z(h_1) $$
   almost surely. Now the assertion is an immediate consequence of the bias condition (B$_b$R), as
	$Z(g_1) - \theta Z(h_1)$ is a centered normal random variable with the same variance $Var(Z(g_1))+\theta^2 Var(Z(h_1))-2 Cov(Z(g_1),Z(h_1))=\theta(\theta c-1)$ as obtained in Theorem 3.1.
\end{proof}

\section{Lindeberg condition (L2) implies conditions (L1) and (L)}

In Appendix A, the Lindeberg condition (L2) for $V_n(\GG):=\sup_{g\in\GG}|V_n(g)|$ was introduced:
\begin{itemize}
	\item[(L2)]
	\begin{equation}
	\frac{m_n}{p_n } E^*\Big[ (V_n(\mathcal{G}))^2 \mathds{1}_{\{V_n(\mathcal{G})>\sqrt{p_n}\epsilon\}}\Big]\to 0, \hspace{2cm}\forall\epsilon>0.
	\end{equation}
\end{itemize}

By the Cauchy-Schwarz inequality and Chebyshev's inequality it follows
\begin{align}
&E^*\left[ V_n(\mathcal{G}) \mathds{1}_{\left\{V_n(\mathcal{G})>\sqrt{p_n}\epsilon\right\}}\right]
\leq \sqrt{E^*\left[ (V_n(\mathcal{G}))^2\mathds{1}_{\left\{V_n(\mathcal{G})>\sqrt{p_n}\epsilon\right\}}\right]E^*\left[\mathds{1}_{\left\{V_n(\mathcal{G})>\sqrt{p_n}\epsilon\right\}}\right]}\\
&\leq \Bigg(\frac{\big(E^*[ (V_n(\mathcal{G}))^2\mathds{1}_{\{V_n(\mathcal{G})>\sqrt{p_n}\epsilon\}}]\big)^2}{\epsilon^2 p_n }\Bigg)^{1/2}\\
&\overset{(L2)}{=}\hbox{o}\bigg(\Big(\frac{p_n^2}{p_nm_n^2}\Big)^{1/2}\bigg)=\hbox{o}\left(\frac{\sqrt{p_n}}{m_n}\right).
\end{align}
Therefore, (L2) implies condition (L1).

Moreover, under (L2), by Jensen's inequality
\begin{equation}
\big(E[V_n(g)]\big)^2
\leq E\big[(V_n(g))^2\big]
\leq E\Big[(V_n(g))^2\mathds{1}_{\left\{|V_n(g)|>\sqrt{p_n}\epsilon\right\}}\Big]+\epsilon^2 p_n =\ord\Big(\frac{p_n}{m_n}\Big)+\epsilon^2 p_n
\end{equation}
for all $\epsilon>0$, and thus $E |V_n(g)|=\ord(\sqrt{p_n})$.
We may conclude for sufficiently large $n$
\begin{align}
  E\Big[ & (V_n(g)-EV_n(g))^2 \ind{\{|V_n(g)-EV_n(g)|>\eps \sqrt{p_n}\}}\Big]\\
   & \le 2 E\Big[ \big((V_n(g))^2+(EV_n(g))^2\big) \ind{\{|V_n(g)|>\eps \sqrt{p_n}/2\}}\Big]\\
   & \le 2 E^*\Big[ (V_n(\mathcal{G}))^2 \mathds{1}_{\{V_n(\mathcal{G})>\epsilon\sqrt{p_n}/2\}}\Big] +\ord(p_n)P\big\{|V_n(g)|>\epsilon\sqrt{p_n}/2\big\} \\
   &\le 4 E^*\Big[ (V_n(\mathcal{G}))^2 \mathds{1}_{\{V_n(\mathcal{G})>\epsilon\sqrt{p_n}/2\}}\Big]\\
   & = \ord\Big(\frac{p_n}{m_n}\Big),
\end{align}
i.e.\ (L) holds.

\addcontentsline{toc}{section}{References}
\bibliography{Dissbib}
\bibliographystyle{agsm}